\documentclass[a4paper,11pt]{amsart}
\usepackage{amsmath}
\usepackage{url,graphicx}
\usepackage{amssymb, color, pstricks-add, manfnt}
\usepackage{amsmath, amsthm, amsfonts}
\usepackage{mathrsfs}
\usepackage{cite}
\usepackage{multirow}
\usepackage{makecell}
\usepackage{diagbox}
\usepackage{tikz}
\usetikzlibrary{matrix, fit}
\usepackage{enumerate}
\usepackage[none]{hyphenat}

\theoremstyle{plain}
\newtheorem{Theorem}{Theorem}[section]
\newtheorem*{Theorem*}{Theorem}
\newtheorem*{Conjecture*}{Conjecture}
\newtheorem{Lemma}{Lemma}[section]

\newtheorem{Corollary}{Corollary}[section]

\theoremstyle{remark}
\newtheorem{remark}{Remark}

\numberwithin{equation}{section}
\numberwithin{figure}{section}
\numberwithin{remark}{section}

\usepackage{autobreak}
\allowdisplaybreaks

\usepackage{hyperref}
\usepackage{cite}

\begin{document}
\begin{sloppypar}

\title{New Newton-type inequality and a unified perspective on Maclaurin's inequality via the inclusion relations of cones}

 \author{Ning Cao}
 \address{School of Mathematics and Shing-Tung Yau Center of Southeast University, Southeast University, Nanjing 211189, P.R. China}
 \email{caoning@seu.edu.cn}	

 \author{Xuanxuan Hu}
 \address{Center for Mathematics and Interdisciplinary Sciences, Fudan University, Shanghai 200433, P.R. China; Shanghai Institute for Mathematics and Interdisciplinary Sciences, Shanghai, 200433, P.R. China}
 \email{25114020008@m.fudan.edu.cn}

 \author{Feida Jiang$^*$}
 \address{School of Mathematics and Shing-Tung Yau Center of Southeast University, Southeast University, Nanjing 211189, P.R. China; Shanghai Institute for Mathematics and Interdisciplinary Sciences, Shanghai, 200433, P.R. China}
 \email{jiangfeida@seu.edu.cn}	
	

	\date{\today}
	\thanks{*corresponding author}
	

	\begin{abstract}
      The classical Newton-Maclaurin's inequalities for the elementary symmetric means are well known. In this paper, we derive a new class of Newton-type inequalities and generalize Maclaurin's inequality by connecting the elementary symmetric functions with their counterparts. Our idea is mainly motivated by the inclusion relations of G$\mathring{\text{a}}$rding's cones and $k$-convex cones.
	\end{abstract}

	\maketitle		

\section{Introduction}\label{Section 1}

The $k$-th elementary symmetric function in $n$ variables $x_1,...,x_n$ is defined by
\begin{equation} \label{sigma}
    \sigma_k(x):= \sum_{1\leq i_1<\cdots<i_k\leq n} x_{i_1}\cdots x_{i_k},
\end{equation}
where $1\leq k\leq n$, and $x=(x_1,...,x_n)\in\mathbb{R}^n$. For convenience, we usually define $\sigma_0(x):=1$ and $\sigma_k(x):=0$ if $k<0$ or $k>n$. Furthermore, the $k$-th elementary symmetric mean is defined by 
\begin{equation*}
    E_k(x) := \frac{\sigma_k(x)}{C_n^k},
\end{equation*}
where $0\leq k\leq n$, and $C_n^k:=\frac{n!}{k!(n-k)!}$ denotes the combination number. The elementary symmetric functions have a wide range of applications in fully nonlinear partial differential equations and can be used to define the $k$-Hessian operators, which is extensively studied in fully nonlinear partial differential equations related to geometric problems, e.g., see \cite{chou2006lp, liu2010brunn, guan2003christoffel}, etc. For more details on $k$-Hessian operators, one can refer to \cite{2009The, caffarelli1985dirichlet, trudinger1995dirichlet}.

There is another class of symmetric functions, less well known but also important, which we call $\Pi_k$ symmetric functions. The $\Pi_k$ symmetric functions can be regarded as the counterpart of the $k$-th elementary symmetric function, and is defined as follows:
\begin{equation} \label{Pi}
    \Pi_k(x):= \prod_{1\leq i_1<\cdots<i_k\leq n} (x_{i_1}+\cdots +x_{i_k}),
\end{equation}
where $1\leq k\leq n$, and $x=(x_1,...,x_n)\in\mathbb{R}^n$. The $\Pi_k$ symmetric functions also play an important role in fully nonlinear partial differential equations and can be used to define the $\Pi_k$-Hessian operators, which are closely related to $k$-convex hypersurfaces in differential geometry, see \cite{Harvey2013, sha1986p, sha1987handlebodies}, etc.

Comparing these two classes of functions, their order of operations is exactly opposite. Moreover, it is obvious that $ \Pi_1(x) = \sigma_n(x) = \prod_{i=1}^n x_i$ and $\Pi_n(x) = \sigma_1(x) = \sum_{i=1}^n x_i$.

Next, we introduce the G$\mathring{\text{a}}$rding's cones $\Gamma_k$ and the $k$-convex cones $\mathcal{P}_k$ in $\mathbb{R}^n$, which are defined as follows:
\begin{equation*}
    \Gamma_k:=\{ x=(x_1,...,x_n)\in\mathbb{R}^n : \sigma_j(x)>0,\ \forall \ 1\leq j\leq k  \},
\end{equation*}
and
\begin{equation*}
    \mathcal{P}_k:=\{ x=(x_1,...,x_n)\in\mathbb{R}^n : x_{i_1}+\cdots+x_{i_k}>0,\ \forall \ 1\leq i_1<\cdots< i_k \leq n \},
\end{equation*}
where $1\leq k\leq n$. It is straightforward that $\Gamma_j\subset\Gamma_k$ and $\mathcal{P}_k\subset\mathcal{P}_j$ whenever $1\leq k<j\leq n$, and $\Gamma_n=\mathcal{P}_1$ is the positive cone, which is denoted by $K^+$ in this paper. The G$\mathring{\text{a}}$rding's cones and $k$-convex cones have significant applications in fully nonlinear partial differential equations. These two classes of cones can provide admissible sets for solutions of the equations, which guarantee the ellipticity of the $k$-Hessian operators and $\Pi_k$-Hessian operators, respectively. For more details on the G$\mathring{\text{a}}$rding's cones and $k$-convex cones as well as their applications to fully nonlinear partial differential equations, we refer the reader to \cite{ivochkina2012garding, ivochkina2014garding, Harvey2013}.

In \cite{li2020inclusion}, Li-Cui-Jiang studied the inclusion relations of the G$\mathring{\text{a}}$rding's cones and the $k$-convex cones, they proved a basic inclusion relations of the G$\mathring{\text{a}}$rding's cone $\Gamma_{n-k+1}$ and the $k$-convex cone $\mathcal{P}_k$, which is $\Gamma_{n-k+1}\subset\mathcal{P}_k$, where $1\leq k\leq n$. Combining the inclusion relations of G$\mathring{\text{a}}$rding's cones and $k$-convex cones themselves, we can obtain a basic network of the inclusion relations of these cones in $\mathbb{R}^n$, as shown in Figure \ref{Fig1.1}.

\begin{figure}[htbp] 
\centering
\[
\begin{array}{ccccccccccccc}
\mathcal{P}_1 & \subset & \mathcal{P}_2 & \subset & \cdots & \subset & \mathcal{P}_{k} & \subset & \cdots & \subset & \mathcal{P}_{n-1} & \subset & \mathcal{P}_n \\[2pt]
\scalebox{1.2}{$\shortparallel$} & & \cup & & \cdots & & \cup & & \cdots & & \cup & & \scalebox{1.2}{$\shortparallel$} \\[1pt]
\Gamma_n & \subset & \Gamma_{n-1} & \subset & \cdots & \subset & \Gamma_{n-k+1} & \subset & \cdots & \subset & \Gamma_2 & \subset & \Gamma_1
\end{array}
\]
\caption{The basic inclusion relations of G$\mathring{\text{a}}$rding's cones and $k$-convex cones}
\label{Fig1.1}
\end{figure}

The inclusion relations shown in Figure \ref{Fig1.1} is the basic starting point for our generalization of the Maclaurin's inequality. But we find that the inclusion relations proved in \cite{li2020inclusion} do not cover all cases. Thus, in this paper, we also clarify all possible inclusion relations among these two classes of cones. Firstly, it is essential to recall the well-known Newton's inequality and Maclaurin's inequality for the elementary symmetric means, which are stated as follows.

\begin{Theorem*} [Newton's inequality, \cite{Newton1707}] \label{Th1.1}
    Let $1\leq k< n$. For any $x=(x_1,...,x_n)\in\mathbb{R}^n$, then the following inequality holds:
\begin{equation} \label{N1}
    E_k^2(x) \geq E_{k-1}(x) E_{k+1}(x),
\end{equation}
    and the equality holds if and only if $x_1 = \cdots = x_n$ or both sides of the inequality are zero values.
\end{Theorem*}

\begin{Theorem*} [Maclaurin's inequality, \cite{Maclaurin1729}] \label{Th1.2}
    Let $1\leq j< l\leq n$. For any $x=(x_1,...,x_n)\in\overline{\Gamma}_l$, then the following inequality holds:
\begin{equation} \label{M1}
     E_j^{\frac{1}{j}}(x) \geq E_l^{\frac{1}{l}}(x),
\end{equation}
    and the equality holds if and only if $x_1 = \cdots = x_n$ or both sides of the inequality are zero values.
\end{Theorem*}

The proof of the Newton's inequality relies on the theory of real roots of polynomials, the complete proof can be found in \cite{hardy1952inequalities} or \cite{steele2004cauchy}. The Newton's inequality is stronger than the Maclaurin's inequality, so the Maclaurin's inequality follows from the Newton's inequality easily, see \cite{hardy1952inequalities}, and also see \cite{steele2004cauchy} for an intuitive geometric proof via the log-concavity. The above Newton-Maclaurin's inequalities provide a fundamental theoretical tool for the study of various problems related to the $k$-Hessian equations, e.g., see \cite{bao2003liouville, brandolini2008overdetermined, du2022bernstein, lin1994some}, etc. For more generalizations and variants of the Newton's inequality and Maclaurin's inequality, we refer the interested readers to \cite{rosset1989normalized, xu2008generalized, favaro2018general, tinaztepe2022refinement, ren2024generalization, ren2025general, bautista2025newton, hu2025newton, tao2025variant, fan2026operators}, etc.

In view of the above theorem on the Maclaurin's inequality, for any $1\leq k\leq n$, we define $E_k^\frac{1}{k}$ as the $k$-th Maclaurin mean corresponding to the G$\mathring{\text{a}}$rding's cone $\Gamma_k$. For any $1\leq j<l\leq n$, by the inclusion relations of G$\mathring{\text{a}}$rding's cones, we have $\Gamma_l\subset\Gamma_j$, the Maclaurin's inequality asserts that the corresponding Maclaurin means $E_j^\frac{1}{j}\geq E_l^\frac{1}{l}$ on $\overline{\Gamma}_l$. Therefore, it seems reasonable to conjecture that Maclaurin's inequality is essentially a class of inequalities concerning the Maclaurin means corresponding to the cones, and determined by the inclusion relations of cones. More generally, it is actually determined by the G$\mathring{\text{a}}$rding's cones and $k$-convex cones together with their own inclusion relations. Based on this consideration, for any $1\leq k\leq n$, we may naturally define the mean $\frac{\Pi_k^{1/C_n^k}}{k}$ as the $k$-th Maclaurin mean corresponding to the $k$-convex cone $\mathcal{P}_k$. For convenience of expression, for any $1\leq k\leq n$, we let
\begin{equation}
    M_k(x) := \frac{\Pi_k^\frac{1}{C_n^k}(x)}{k}, \quad x\in\mathcal{P}_k.
\end{equation}
Our conjecture can be stated precisely as follows.

\begin{Conjecture*}
Let $n\geq2$ and $C_1$, $C_2$ be any G$\mathring{\text{a}}$rding's cone or $k$-convex cone in $\mathbb{R}^n$, and let $Q_1$ and $Q_2$ denote the Maclaurin means corresponding to $C_1$ and $C_2$, respectively. If $C_1\subset C_2\ (C_2\subset C_1)$, then $Q_1(x)\leq(\geq)\, Q_2(x)$ on $\overline{C_1}\,(\overline{C_2})$. Otherwise, there is no inequality relation between $Q_1$ and $Q_2$ on $\overline{K^+}$.
\end{Conjecture*}

The above conjecture is the generalization of the classical Maclaurin's inequality via the inclusion relations of cones. In the current paper, our main purpose is to establish a Newton-type inequality for the $\Pi_k$ symmetric functions. Furthermore, we confirm the above conjecture by analyzing all possible cases. However, we do not know why the inclusion relations of the cones determines the inequality relation between the corresponding Maclaurin means. Our main results are as follows.
\begin{Theorem} [Newton-type inequality] \label{Th1.3}
Let $2\leq k\leq n-1$. For any $x=(x_1,...,x_n)\in\overline{\mathcal{P}}_{k-1}$, then the following inequality holds:
\begin{equation} \label{N2}
    M_k^2(x) \geq M_{k-1}(x) M_{k+1}(x),
\end{equation}
and the equality holds if and only if $x_1=\cdots=x_n>0$, or $x$ contains at least $k$ zeros and all the remaining coordinates are nonnegative.
\end{Theorem}

The Newton-type inequality is essentially a class of log-concavity inequalities, which implies the concavity of $\log M_k(x)$ with respect to $k$. From this, we can easily obtain the following interpolation inequality.

\begin{Corollary} [Interpolation inequality] \label{Co1.1}
Let $1\le p<r<s\le n$. For any $x\in\overline{\mathcal{P}}_p$, we have
\begin{equation} \label{Interpolation}
 M_{r}(x) \ge M_{p}^{\frac{s-r}{s-p}}(x)\cdot M_{s}^{\frac{r-p}{s-p}}(x).
\end{equation}
\end{Corollary}

\begin{remark}
    In fact, the above interpolation inequality \eqref{Interpolation} is a generalization of an inequality in \cite{ito2008mixed}, where the case $p=1$ and $s=n$ was proved.
\end{remark}

As mentioned earlier, the classical Maclaurin's inequality \eqref{M1} follows easily from the Newton's inequality \eqref{N1}. Similarly, as a direct consequence of Newton-type inequality \eqref{N2}, we can also obtain the following Maclaurin-type inequality.

\begin{Corollary} \label{Co1.2}
    Let $1\leq j< l\leq n$. For any $x=(x_1,...,x_n)\in\overline{\mathcal{P}}_j$, then the following inequality holds:
\begin{equation} \label{M2}
    M_l(x) \geq M_j(x),
\end{equation}
    and the equality holds if and only if $x_1 = \cdots = x_n>0$, or $x$ contains at least $l$ zeros and all the remaining coordinates are nonnegative.
\end{Corollary}

Corollary \ref{Co1.2} is the Maclaurin-type inequality concerning the Maclaurin means corresponding to the inclusion relations $\mathcal{P}_j\subset\mathcal{P}_l$. In \cite{favaro2018general}, the authors proved a general Maclaurin's inequality with weights by using inequalities between Jacobi polynomials, their result is a weighted generalization of the classical Maclaurin's inequality. Soon after, in \cite{favaro2018addendum}, they further presented a very short proof of this general inequality. The simple proof is based on a combinatorial observation and the use of the classical Maclaurin's inequality. Since we have derived inequality \eqref{M2}, the similar idea is applicable. Thus, based on Corollary \ref{Co1.2}, we can easily obtain the following weighted inequality.

\begin{Corollary} \label{Co1.3}
    Let $n\geq2$ and $(w_1,...,w_n)$ be a sequence of positive real numbers satisfying $\sum_{i=1}^n w_i=1$. If $r=(r_1,...,r_n)$ such that for every $1\leq i\leq n$, $r_i = mw_i$ being an integer for some $m\in\mathbb{N}_+$, then for $1\leq j<l\leq m$ and any $x=(x_1,...,x_n)\in\overline{K^+}$, i.e., $x_i\geq0$, the following weighted Maclaurin-type inequality holds:
   \begin{equation} \label{wighted M2}
     \frac{P_l^{\frac{1}{C_m^l}}(x,r)}{l} \geq \frac{P_j^{\frac{1}{C_m^j}}(x,r)}{j},
   \end{equation}
   where 
   $$ 
   P_l(x,r):= \prod_{(l_1,...,l_n)\in\mathcal{S}_{n,l}} \left( \sum_{i=1}^n l_ix_i \right)^{\prod_{i=1}^n C_{r_i}^{l_i}} ,
   $$
   with $\mathcal{S}_{n,l}:=\{ (l_1,...,l_n) \in\mathbb{N}^n: \ \sum_{i=1}^n l_i=l \}$.
   Furthermore, the equality in \eqref{wighted M2} holds if and only if $x_1 = \cdots = x_n>0$ or \ $\sum_{i=1}^n r_i\text{\large$\chi$}_{\{0\}}(x_i)\geq l$, where $\text{\large$\chi$}_{\{0\}}$ denotes the characteristic function relative to set $\{0\}$.
\end{Corollary} 

In view of the cone inclusion relation $\Gamma_{n-k+1}\subset\mathcal{P}_k$, $1\leq k\leq n$, we prove the following inequality between corresponding Maclaurin means $E_{n-k+1}^\frac{1}{n-k+1}$ and $M_k$.

\begin{Theorem} \label{Th1.4}
    Let $n\geq3$ and $1\leq k\leq n$, then for any $x\in\overline{\Gamma}_{n-k+1}$, the following inequality holds:
    \begin{equation} \label{M3}
        M_k(x) \geq E_{n-k+1}^\frac{1}{n-k+1}(x),
    \end{equation}
    when $2\leq k\leq n-1$, the equality holds if and only if $x_1=\cdots=x_n>0$, or $x$ contains at least $k$ zeros and all the remaining coordinates are nonnegative.
\end{Theorem}

Combining the classical Maclaurin's inequality and Theorem \ref{Th1.4}, we immediately obtain the following corollary.

\begin{Corollary} \label{Co1.4}
    Let $n\geq3$, $1\leq k,l\leq n$ such that $k+l\geq n+1$, then for any $x\in \overline{\Gamma}_{l}$, the following inequality holds:
    \begin{equation} \label{M3 kl}
        M_k(x) \geq E_{l}^\frac{1}{l}(x),
    \end{equation}
    when $2\leq k\leq n-1$, the equality holds if and only if $x_1=\cdots=x_n>0$, or $x$ contains at least $k$ zeros and all the remaining coordinates are nonnegative.
\end{Corollary}

\begin{remark} \label{remark 1.2}
    Clearly, under the condition of Corollary \ref{Co1.4}, we also have $\Gamma_l\subset\mathcal{P}_k$ in $\mathbb{R}^n$.
\end{remark}

\begin{remark}
        In fact, the above inequality \eqref{M3 kl} is a conjecture proposed in \cite{carlson1971mixed}, where the authors conjectured that it holds on $\overline{K^+}$, based on an observation different from ours. Later, this conjecture was proved in \cite{ito2008mixed}. Although only on $\overline{K^+}$, the proof in \cite{ito2008mixed} was already quite complicated. In this paper, we prove that this inequality holds on $\overline{\Gamma}_{l}$, not merely on $\overline{K^+}$, based on a simpler proof. 
\end{remark}

In addition, we also prove the following inequality between two Maclaurin means $E_k^\frac{1}{k}$ and $M_r$, for any $n\geq k r$ with $1<k,r<n$. This result is motivated by the cone inclusion relation $\mathcal{P}_2\subset\Gamma_2$ in $\mathbb{R}^4$, which was proved in \cite{li2020inclusion}.

\begin{Theorem} \label{Th1.5}
    Let $n\geq kr$ with $1<k,r<n$. Then for any $x=(x_1,...,x_n)\in \overline{\mathcal{P}}_r$, the following inequality holds:
    \begin{equation} \label{M4}
        E_k^{\frac{1}{k}}(x) \geq M_r(x).
    \end{equation}
    If $n=kr$, then equality holds if and only if at least $n-1$ coordinates of $x$ are equal, or $x$ contains at least $n-k+1$ zeros and all the remaining coordinates are nonnegative. If $n>kr$, then equality holds if and only if $x_1=\cdots=x_n>0$, or $x$ contains at least $n-k+1$ zeros and all the remaining coordinates are nonnegative. 
\end{Theorem}

From the proof of Theorem \ref{Th1.5}, we can also obtain the corresponding inclusion relation of cones $\Gamma_k$ and $\mathcal{P}_r$.

\begin{Corollary} \label{Co1.5}
    Let $n\geq kr$ with $1<k,r<n$. Then $\mathcal{P}_r \subset \Gamma_k$ in $\mathbb{R}^n$.
\end{Corollary}

\begin{remark}
    In \cite{li2020inclusion}, the authors proved that $\mathcal{P}_2\subset \Gamma_2$ in $\mathbb{R}^4$, which is a special case of Corollary \ref{Co1.5}. Thus, Corollary \ref{Co1.5} can be regarded as a generalization of the result in \cite{li2020inclusion}. 
\end{remark}
 
Combining the classical Maclaurin's inequality, Corollary \ref{Co1.2} and Theorems \ref{Th1.4}, \ref{Th1.5}, we can easily obtain the complete inequality chain on $\overline{K^+}$ between all Maclaurin means when $n=3,4$. However, when $n\geq5$, the complete inequality chain on $\overline{K^+}$ between all Maclaurin means generally does not exist, see Remark \ref{n geq 5} in Section \ref{Section 3}.

\begin{Corollary} \label{Co1.6}
     Let $n=3$, then for any $x=(x_1,x_2,x_3)\in \overline{K^+}$, the following inequality chain holds:
    \begin{equation} \label{n=3}
       M_3(x) = E_1(x) \geq M_2(x) \geq E_2^\frac{1}{2}(x) \geq M_1(x) = E_3^\frac{1}{3}(x).
    \end{equation}
\end{Corollary} 
   
\begin{Corollary} \label{Co1.7}
    Let $n=4$, then for any $x=(x_1,x_2,x_3,x_4)\in \overline{K^+}$, the following inequality chain holds:
    \begin{equation} \label{n=4}
        M_4(x) = E_1(x) \geq M_3(x) \geq E_2^\frac{1}{2}(x) \geq M_2(x) \geq E_3^\frac{1}{3}(x) \geq M_1(x) = E_4^\frac{1}{4}(x).
    \end{equation}
\end{Corollary}

\begin{remark}
    Of course, for $n=3,4$, we have the corresponding inclusion relations of the cones associated with inequality chains \eqref{n=3} and \eqref{n=4}, which are
    \begin{equation*}
        \Gamma_3 = \mathcal{P}_1 \subset \Gamma_2 \subset\mathcal{P}_2\subset \Gamma_1 = \mathcal{P}_3 \quad \text{in} \quad \mathbb{R}^3,
    \end{equation*}
    and
    \begin{equation*}
      \Gamma_4 = \mathcal{P}_1 \subset \Gamma_3 \subset \mathcal{P}_2 \subset \Gamma_2 \subset\mathcal{P}_3\subset \Gamma_1 = \mathcal{P}_4 \quad \text{in} \quad \mathbb{R}^4,
    \end{equation*}
    see \cite{li2020inclusion}.
\end{remark}

In general, for $n\geq5$, combining Remark \ref{remark 1.2}, Corollary \ref{Co1.5} and Remark \ref{n geq 5}, we can obtain the complete inclusion relations of G$\mathring{\text{a}}$rding's cones and $k$-convex cones in $\mathbb{R}^n$, as shown in Figure \ref{Fig1.2}.

\begin{figure}[htbp]
  \centering
  \includegraphics[width=0.9\textwidth]{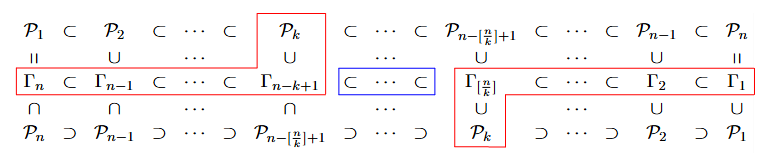} 
  \vspace{-10pt}
  \caption{The inclusion relations of G$\mathring{\text{a}}$rding's cones and $k$-convex cones $(n\geq5)$} 
  \label{Fig1.2}
\end{figure}

In Figure \ref{Fig1.2}, for the $(n-1)$-convex cone $\mathcal{P}_{n-1}$, its inclusion relations with the other cones are obvious; for the $k$-convex cone $\mathcal{P}_k$, where $2\leq k\leq n-2$, the red box indicates that it has an inclusion relation with some of the G$\mathring{\text{a}}$rding's cones, but with none of the remaining G$\mathring{\text{a}}$rding's cones in the blue box. This provides a complete description of the inclusion relations among these cones. Moreover, the complete inequality relations between the Maclaurin means associated with the inclusion relations of G$\mathring{\text{a}}$rding's cones and $k$-convex cones shown in Figure \ref{Fig1.2} are also obtained by combining inequalities \eqref{M1}, \eqref{M2}, \eqref{M3}, \eqref{M4} and Remark \ref{n geq 5}, as shown in Figure \ref{Fig1.3}. Analogously, the red box indicates that the Maclaurin mean $M_k$ has an inequality relation with some of the Maclaurin means corresponding to the G$\mathring{\text{a}}$rding's cones, but has no inequality relation on $\overline{K^+}$ with any of the remaining Maclaurin means corresponding to the G$\mathring{\text{a}}$rding's cones in the blue box. This confirms the conjecture proposed in this paper. Comparing the inclusion relations of the cones with the inequality relations among the Maclaurin means, it is interesting that the two are strikingly consistent. This may suggest something deeper connection. By the way, we believe that the inequalities proved in this paper can be applied to the study of problems related to the $k$-Hessian equations and $\Pi_k$-Hessian equations.

\begin{figure}[htbp]
  \centering
  \includegraphics[width=0.9\textwidth]{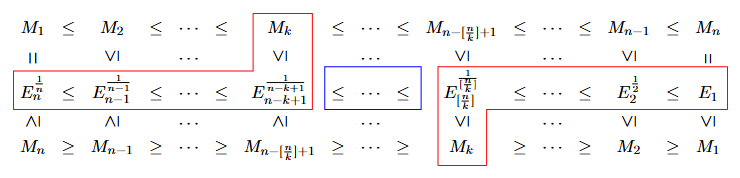} 
  \vspace{-10pt}
  \caption{The generalized chain of Maclaurin's inequalities $(n\geq5)$} 
  \label{Fig1.3}
\end{figure}

This paper is organized as follows. In Section \ref{Section 2}, we first prove a local logarithmic inequality which is key to the proof of Theorem \ref{Th1.3}. Moreover, we present some useful properties of the elementary symmetric functions and prove a key inequality, which will be employed in the subsequent proofs. In Section \ref{Section 3}, we provide the proofs of Theorems \ref{Th1.3}, \ref{Th1.4}, \ref{Th1.5} and the related corollaries, respectively.

\vspace{3mm}

\section{Preliminaries} \label{Section 2}

In this section, we introduce some preliminary results, which mainly involve a local logarithmic inequality, some basic facts about the elementary symmetric functions, and a key inequality concerning the Maclaurin means. These results will be used in Section \ref{Section 3}.

\subsection{A local logarithmic inequality}
In this subsection, we prove a local logarithmic inequality on a certain domain contained in some affine hyperplane, which plays an important role in the proof of Theorem \ref{Th1.3}. 

For any given integer $q\ge3$, we define the domain:
\begin{equation} \label{local-domain}
 \mathcal D_q:=\left\{ x=(x_1,\ldots,x_q)\in\mathbb R^q:
 \sum_{i=1}^q x_i=1,\ \
 1-x_i-x_j>0,\, \forall 1\le i<j\le q \right\}.
\end{equation}
Clearly, the set $\mathcal D_q$ is relatively open and convex in the affine hyperplane $\sum_{i=1}^q x_i=1$. Moreover, we denote the vector
$x^q := (1/q,\ldots,1/q)\in\mathbb R^q$. Obviously, we have $x^q\in\mathcal D_q$. To prove Theorem \ref{Th1.3}, we need to study a specific function defined on the domain $\mathcal{D}_q$. We define the function
\begin{equation}
    \mathcal L_q(x)
 :=(q-1)\sum_{i=1}^q\log(1-x_i) - \sum_{1\le i<j\le q}\log(1-x_i-x_j), \quad x\in\mathcal{D}_q.
\end{equation}
Firstly, we note that the function $\mathcal{L}_q$ is well defined on $\mathcal{D}_q$, we need to show that $1-x_i>0$ for any $1\leq i\leq q$ on $\mathcal{D}_q$. To see this, for any fixed $1\leq i\leq q$, since $\sum_{j=1}^q x_j=1$, so we have
\begin{equation} \label{well-defined}
\begin{aligned}
 \sum_{\substack{j=1 \\ j\neq i}}^q(1-x_i-x_j)
 &=(q-1)(1-x_i)-\sum_{\substack{j=1 \\ j\neq i}}^q x_j\\
 &=(q-1)(1-x_i)-(1-x_i)\\
 &=(q-2)(1-x_i).
\end{aligned}
\end{equation}
Every summand in the first line of \eqref{well-defined} is positive and $q-2>0$, thus $1-x_i>0$.
The function $\mathcal{L}_q$ is the key to our proof of Theorem \ref{Th1.3}, and we have the following lemmas about the function $\mathcal{L}_q$.

\begin{Lemma} \label{Lemma2.1}
    Let $q\geq3$. For any given $x\in\mathcal{D}_q$ and $x\neq x^q$, let $i\neq j$ such that 
    \begin{equation}
        x_i=\min_{1\leq l\leq q}x_l,\quad x_j=\max_{1\leq l\leq q} x_l.
    \end{equation}
    Then we have
    \begin{equation}
        \frac{\partial\mathcal{L}_q}{\partial x_i}(x) < \frac{\partial\mathcal{L}_q}{\partial x_j}(x).
    \end{equation}
\end{Lemma}

\begin{proof}
For convenience, we let $a:=x_i=\min_{1\leq l\leq q}x_l$ and $b:=x_j=\max_{1\leq l\leq q}x_l$. Clearly, we have $a<b$, since $x\neq x^q$. Moreover, we define the average of the other $q-2$ coordinates of $x\in\mathcal{D}_q$ by
\begin{equation}\label{c-definition}
 c:=\frac1{q-2}\sum_{\substack{l=1 \\l\neq i,j}}^q x_l
 =\frac{1-a-b}{q-2}.
\end{equation}
By the definition of $a,b$ and $\mathcal{D}_q$, we have $c>0$ and $a\le c\le b$.

For any $1\leq k\leq q$, by direct calculation, we have
\begin{equation}\label{eq:partial}
 \frac{\partial\mathcal L_q}{\partial x_k}(x)
 =-\frac{q-1}{1-x_k}
 +\sum_{\substack{l=1 \\ l\ne k}}^q\frac1{1-x_k-x_l}.
\end{equation}
Thus, for $k=i,j$, we have
\begin{equation}
\begin{aligned}
    \frac{\partial\mathcal L_q}{\partial x_i}(x)
 -\frac{\partial\mathcal L_q}{\partial x_j}(x) &= -(q-1)\left(\frac1{1-a}-\frac1{1-b}\right)
 +\sum_{\substack{l=1\\ l\neq i,j}}^q
 \left(\frac1{1-a-x_l}-\frac1{1-b-x_l}\right) \\
 &= -(q-1)\frac{a-b}{(1-a)(1-b)} + \sum_{\substack{l=1\\ l\neq i,j}}^q \frac{a-b}{(1-a-x_l)(1-b-x_l)}.
\end{aligned}
\end{equation}
Since $a-b\ne0$, we can obtain
\begin{equation} \label{E1}
 \frac{1}{a-b}\left(
 \frac{\partial\mathcal L_q}{\partial x_i}(x)
 -\frac{\partial\mathcal L_q}{\partial x_j}(x)\right) = \sum_{\substack{l=1 \\ l\neq i,j}}^q
 \frac{1}{(1-a-x_l)(1-b-x_l)}
 -\frac{q-1}{(1-a)(1-b)}.
\end{equation}
For fixed real numbers $a<b$ above, we let
\[
 h_{a,b}(s) := \frac1{(1-a-s)(1-b-s)}, \quad s\in(-\infty,1-b).
\]
Next, we compute the second derivative of $h_{a,b}$,
\[
 h_{a,b}'(s)
 =\frac1{(1-a-s)^2(1-b-s)} + \frac1{(1-a-s)(1-b-s)^2},
 \]
 thus
 \[
 h_{a,b}''(s) = \frac2{(1-a-s)^3(1-b-s)} + \frac2{(1-a-s)^2(1-b-s)^2} + \frac2{(1-a-s)(1-b-s)^3}>0,
\]
since $1-a-s>1-b-s>0$. Therefore, the function $h_{a,b}$ is strictly convex on $(-\infty,1-b)$. For every $l\neq i,j$, we have $1-b-x_l>0$, i.e., $x_l<1-b$, so we can obtain that $c<1-b$. The convexity of $h_{a,b}$ therefore gives
\begin{equation} \label{Con}
  h_{a,b}(x_l)\ge h_{a,b}(c)+h_{a,b}'(c)(x_l-c).   
\end{equation}
Summing \eqref{Con} over all $l\neq i,j$, and because $\sum_{l=1, l\neq i,j}^q x_l=(q-2)c$, we can obtain
\begin{equation} \label{convexity-estimate}
 \sum_{\substack{l=1 \\ l\neq i,j}}^q
 \frac1{(1-a-x_l)(1-b-x_l)}
 \ge\frac{q-2}{(1-a-c)(1-b-c)}.
\end{equation}
Thus by \eqref{E1} and \eqref{convexity-estimate}, we have
\begin{equation} \label{Ineq 1}
\begin{aligned}
    \frac{1}{a-b}\left(\frac{\partial\mathcal L_q}{\partial x_i}(x) -\frac{\partial\mathcal L_q}{\partial x_j}(x)\right) &\geq \frac{q-2}{(1-a-c)(1-b-c)}-\frac{q-1}{(1-a)(1-b)} \\
    &= \frac{(q-2)(1-a)(1-b)-(q-1)(1-a-c)(1-b-c)}{(1-a)(1-b)(1-a-c)(1-b-c)}.
\end{aligned}
\end{equation}
From \eqref{c-definition}, we have $a+b=1-(q-2)c$. Then by direct calculation,
\begin{equation} \label{E2}
\begin{aligned}
 &\quad\ (q-2)(1-a)(1-b)-(q-1)(1-a-c)(1-b-c) \\
 &= -(1-a)(1-b)+(q-1)(2-a-b)c-(q-1)c^2 \\
 &=-1+a+b-ab+(q-1)(2-a-b)c-(q-1)c^2 \\
 &=-(q-2)c-ab+(q-1)\big(1+(q-2)c\big)c-(q-1)c^2 \\
 &=c-ab+(q-1)(q-3)c^2.
\end{aligned}
\end{equation}
On the other hand,
\begin{equation} \label{E3}
\begin{aligned}
 &\quad\ (q-1)(q-2)c^2+(c-a)(b-c) \\
 &=(q-1)(q-2)c^2+c(a+b)-c^2-ab \\
 &=(q-1)(q-2)c^2+c\bigl(1-(q-2)c\bigr)-c^2-ab \\
 &=c-ab+(q-1)(q-3)c^2.
\end{aligned}
\end{equation}
Thus, by \eqref{E2} and \eqref{E3}, we obtain
\begin{equation} \label{E4}
\begin{aligned}
 &\quad\ (q-2)(1-a)(1-b)-(q-1)(1-a-c)(1-b-c) \\
 &= (q-1)(q-2)c^2+(c-a)(b-c)>0,
\end{aligned}
\end{equation}
since $c>0$, $q\ge3$, and $a\leq c\leq b$. Furthermore, from the previous discussion, we have
\begin{equation} \label{Ineq 2}
    (1-a)(1-b)(1-a-c)(1-b-c)>0.
\end{equation}
Combining \eqref{Ineq 1}, \eqref{E4} and \eqref{Ineq 2}, we have 
\begin{equation*}
\begin{aligned}
 &\quad\ \frac1{a-b}\left(
 \frac{\partial\mathcal L_q}{\partial x_i}(x)
 -\frac{\partial\mathcal L_q}{\partial x_j}(x)\right) \ge\frac{(q-1)(q-2)c^2+(c-a)(b-c)}{(1-a)(1-b)(1-a-c)(1-b-c)}>0.\label{eq:strict-derivative-sign}
\end{aligned}
\end{equation*}
Since $a-b<0$, it follows that
\begin{equation}
 \frac{\partial\mathcal L_q}{\partial x_i}(x) -\frac{\partial\mathcal L_q}{\partial x_j}(x) < 0.
\end{equation}
Thus, we complete the proof of this lemma.
\end{proof}
By employing the above lemma, we can prove the function $\mathcal{L}_q$ admit minimum on $\mathcal{D}_q$.

\begin{Lemma}[] \label{Lemma2.2}
Let $q\ge3$. The function $\mathcal{L}_q$ satisfies
\begin{equation}\label{local-minimum}
 \mathcal L_q(x)\ge \mathcal{L}_q(x^q) = C_q^2\log\left(\frac{(q-1)^2}{q(q-2)}\right) \quad \text{on} \quad \mathcal{D}_q,
\end{equation}
where equality holds if and only if $x=x^q$. That is, the function $\mathcal{L}_q$ attains its unique minimum at $x^q$ on $\mathcal{D}_q$.
\end{Lemma}

\begin{proof}
Obviously, the function $\mathcal{L}_q$ is continuous on the domain $\mathcal{D}_q$, but $\mathcal{D}_q$ is not compact, so we cannot assert whether $\mathcal{L}_q$ attains a minimum on $\mathcal{D}_q$. To prove that the minimum of $\mathcal{L}_q$ on the domain $\mathcal{D}_q$ indeed exists and is attained only at $x^q$, we employ a compactness argument together with Lemma \ref{Lemma2.1}, which suffices to yield the desired result. 

To do this, we fix any $x^{0}=(x_1^0,...,x_q^0)\in\mathcal D_q$. Let $S_q$ denotes the symmetric group on $q$ letters. For any $\pi\in S_q$, we set
\[
 x_\pi^{0}
 :=(x^{0}_{\pi(1)},...,x^{0}_{\pi(q)})\in \mathbb{R}^q,
\]
and $\mathcal{H}(x^0)$ denotes the convex hull of all vectors obtained by permuting the coordinates of $x^0$, that is
\begin{equation}
 \mathcal H(x^{0})
 :=\left\{ \sum_{\pi\in S_q}t_\pi x_\pi^{0} \in\mathbb{R}^q : \sum_{\pi\in S_q} t_\pi=1,\ t_\pi\ge0, \forall \pi\in S_q \right\}.
\end{equation}
Obviously, we have $x^0\in\mathcal H(x^{0})$. The coefficients $(t_\pi)_{\pi\in S_q}$ range over a bounded closed set in $\mathbb R^{q!}$, which is compact, and the mapping $(t_\pi)_{\pi\in S_q} \mapsto \sum_{\pi\in S_q}t_\pi x_\pi^{0}$ is clearly continuous. Thus $\mathcal H(x^{0})$ is compact. Furthermore, $\mathcal{H}(x^0)$ is convex by the definition of convex hull, and is invariant under every coordinate permutation, that is, for any $x=(x_1,...,x_q)\in\mathcal{H}(x^0)$ and $\pi\in S_q$, we have $x_\pi=(x_{\pi(1)},...,x_{\pi(q)})\in\mathcal{H}(x^0)$. We also have $x^q\in\mathcal{H}(x^0)$. To see this, we take $t_\pi=1/{q!}$ for every $\pi\in S_q$. Then for any fixed $1\leq r\leq q$, we have
\[
 \frac{1}{q!}\sum_{\pi\in S_q}x^{0}_{\pi(r)} = \frac{(q-1)!}{q!}\sum_{l=1}^q x^{0}_l = \frac1q.
\]
In fact, $\mathcal{H}(x^0)$ is contained in $\mathcal{D}_q$, let
\[
 \delta := \min_{1\le i<j\le q} \big(1-x^{0}_i-x^{0}_j\big)>0.
\]
For any $x=\sum_{\pi\in S_q}t_\pi x_\pi^{0} \in \mathcal H(x^{0})$ , we have 
\[
\sum_{i=1}^q x_i= \sum_{i=1}^q\sum_{\pi\in S_q}t_\pi x_{\pi(i)}^{0} = \sum_{\pi\in S_q}t_\pi \sum_{i=1}^q x_{\pi(i)}^0 = \sum_{i=1}^q x_i^0=1.
\]
And for any $1\le i<j\le q$,
\begin{align}
 1-x_i-x_j = \sum_{\pi\in S_q}t_\pi \bigl(1-x^{0}_{\pi(i)}-x^{0}_{\pi(j)}\bigr) \ge \delta\sum_{\pi\in S_q}t_\pi=\delta>0.
\end{align}
Therefore, $\mathcal H(x^{0})\subset\mathcal D_q$. Since the function $\mathcal L_q$ is continuous on $\mathcal H(x^{0})$, and $\mathcal H(x^{0})$ is compact, hence $\mathcal{L}_q$ attains its minimum on $\mathcal H(x^{0})$.

Let $x\in\mathcal H(x^{0})$ be a minimizer of $\mathcal{L}_q$. We shall prove that $x$ must be $x^q$. Suppose that $x\neq x^q$, let $i\neq j$ such that $x_i=a=\min_{1\le l \le q}x_l$ and $x_j=b=\max_{1\le l\le q}x_l$ as before, so $a<b$. Let $x'$ be obtained from $x$ by exchanging the $i$-th and $j$-th coordinates and leaving all other coordinates fixed. Then the permutation invariance and convexity of $\mathcal H(x^{0})$ imply that
\[
 x(\theta) := (1-\theta)x + \theta x' \in\mathcal H(x^{0}), \quad \forall\ 0\le\theta\le1.
\]
In particular, the coordinates of $x(\theta)$ are $x_i(\theta) = a+\theta(b-a)$, $x_j(\theta) = b-\theta(b-a)$, and $x_l(\theta) = x_l$ for any $l\ne i,j$. Then the chain rule gives
\[
 \left.\frac{d}{d\theta}\mathcal L_q \big( x(\theta) \big)\right|_{\theta=0} = (b-a)\left( \frac{\partial\mathcal L_q}{\partial x_i}(x)
 -\frac{\partial\mathcal L_q}{\partial x_j}(x)\right)<0,
\]
where Lemma \ref{Lemma2.1} is applied to $x$. Thus, we can obtain that $\mathcal L_q\big( x(\theta) \big) < \mathcal L_q(x)$ for all sufficiently small $\theta>0$. This contradicts the fact that $x$ is a minimizer of $\mathcal{L}_q$ on $\mathcal H(x^{0})$, therefore $x$ must be $x^q$. This proves uniqueness of the minimizer of $\mathcal{L}_q$ on the compact set $\mathcal H(x^{0})$.

Because $x^{0}\in\mathcal H(x^{0})$, we have $\mathcal{L}_q(x^0)\geq \mathcal{L}_q(x^q)$, and the equality holds if and only if $x^0=x^q$. Since $x^{0}\in\mathcal{D}_q$ is arbitrary, the lemma is proved.
\end{proof}

\subsection{Some properties of $\sigma_k$} In this subsection, we present some useful properties related to the elementary symmetric functions, based on these properties, we prove a key inequality concerning the Maclaurin means, which will be used in the proof of Theorem \ref{Th1.4}.

Firstly, we introduce the notation $\sigma_k(x|i_1\cdots i_s)$ for any $1\leq k\leq n$ and $s$-tuple $\{i_1,...,i_s\}\subset\{1,...,n\}$, which is defined as follows:
\begin{equation}
    \sigma_k(x|i_1\cdots i_s) := \sigma_k(x)|_{x_{i_1}=\cdots=x_{i_s}=0}, 
\end{equation}
that is, $\sigma_k(x|i_1\cdots i_s)$ is the $k$-th elementary symmetric function of the $n-s$ variables except for $x_{i_1},...,x_{i_s}$. Next, we give some useful properties of $\sigma_k$.

\begin{Lemma} \label{Lemma2}
    Let $1\leq k\leq n$ and $x\in\mathbb{R}^n$, then we have
    \begin{equation} \label{sum property}
        \sigma_k(x) = \sigma_k(x|i) + x_i \sigma_{k-1}(x|i).
    \end{equation}
    Moreover, $\sigma_k^\frac{1}{k}$ is concave on $\Gamma_k$, and for any $x\in\Gamma_k$, we have
    \begin{equation} \label{Euler homogeneity}
        \sum_{i=1}^n x_i \frac{\partial \sigma_k^\frac{1}{k}}{\partial x_i}(x) = \sigma_k^\frac{1}{k}(x).
    \end{equation}
\end{Lemma}

In the above lemma, property \eqref{sum property} is obvious. For the concavity of $\sigma_k^\frac{1}{k}$ on $\Gamma_k$, one can refer to \cite{trudinger1995dirichlet}, and \eqref{Euler homogeneity} is a direct consequence of Euler's homogeneity theorem, since $\sigma_k^\frac{1}{k}$ is a homogeneous function of degree $1$ on $\Gamma_k$.

\begin{Lemma} \label{Lemma3}
    Let $1\leq k\leq n$ and $x\in\Gamma_k$, then for any $\{i_1,...,i_s\}\subset\{1,...,n\}$ and $l+s\leq k$, we have
    \begin{equation}
        \sigma_l(x|i_1\cdots i_s) > 0.
    \end{equation} 
\end{Lemma}

 For Lemma \ref{Lemma3}, one can refer to \cite{lin1994some}, and see \cite{li2020inclusion} for a detailed proof of this property. Finally, we prove a key inequality which will be used in the proof of Theorem \ref{Th1.4}, based on Lemma \ref{Lemma2} and Lemma \ref{Lemma3}.

 \begin{Lemma} \label{Lemma4}
      Let $2\leq m \leq n-1$ and $x=(x_1,...,x_n)\in\Gamma_m$, then we have
     \begin{equation} \label{2.10}
        \left( \prod_{i=1}^n E_{m-1}^\frac{1}{m-1}(\hat{x}^i) \right)^\frac{1}{n} \geq E_{m}^\frac{1}{m}(x),
     \end{equation}
     where $\hat{x}^i:=(x_1,...,x_{i-1},x_{i+1},...,x_n)\in\mathbb{R}^{n-1}$ denotes the $(n-1)$-dimensional vector obtained by deleting the $i$-th coordinate, and the equality holds if and only if $x_1=\cdots=x_n>0$.
 \end{Lemma}
 \begin{proof}
     By Lemma \ref{Lemma2}, the function $f(x):= E_m^\frac{1}{m}(x)$ is concave on $\Gamma_m$. Thus by the basic property of concave function, for any $x,y\in\Gamma_m$, we have
     \begin{equation}
         f(y)\leq f(x) + \nabla f(x)\cdot(y-x) = f(x) + \sum_{i=1}^n\frac{\partial f}{\partial x_i}(x)(y_i-x_i).
     \end{equation}
     By property \eqref{Euler homogeneity}, we can obtain
     \begin{equation} \label{2.12}
         f(y)\leq \sum_{i=1}^n \frac{\partial f}{\partial x_i}(x)y_i.
     \end{equation}
     By direct calculation, we have
     \begin{equation}
         \frac{\partial f}{\partial x_i}(x)=\frac{1}{mC_n^m}\frac{\sigma_{m-1}(x|i)}{f^{m-1}(x)}=\frac{1}{n}\frac{E_{m-1}(\hat{x}^i)}{f^{m-1}(x)}.
     \end{equation}
     Then from \eqref{2.12}, we have
     \begin{equation} \label{2.14}
        f^{m-1}(x) f(y) \leq \frac{1}{n} \sum_{i=1}^n E_{m-1}(\hat{x}^i) y_i, \quad \forall x,y\in\Gamma_m.
     \end{equation}
     To obtain \eqref{2.10}, we need to normalize the right-hand side of \eqref{2.14}. By Lemma \ref{Lemma3}, for any $x\in\Gamma_m$, we have $E_{m-1}(\hat{x}^i) = \frac{\sigma_{m-1}(\hat{x}^i)}{C_{n-1}^{m-1}} >0$, so we can take $y:=(y_1,...,y_n)\in \Gamma_m$, where $y_i=\frac{1}{E_{m-1}(\hat{x}^i)}>0$ for any $1\leq i\leq n$, then from \eqref{2.14}, we obtain
     \begin{equation} \label{2.15}
         f^{m-1}(x) f(y) \leq 1, \quad \forall x\in\Gamma_m.
     \end{equation}
     On the other hand, by the AM-GM inequality, we have
     \begin{equation} \label{AM-GM}
       f^m(y) = \frac{\sigma_m(y)}{C_n^m} \geq \left( \prod_{1\leq i_1<\cdots<i_m\leq n} y_{i_1}\cdots y_{i_m} \right)^\frac{1}{C_n^m} = \left( \prod_{i=1}^ny_i^{C_{n-1}^{m-1}} \right)^\frac{1}{C_n^m} = \left( \prod_{i=1}^n y_i \right)^\frac{m}{n}.
     \end{equation}
     Then, we have
     \begin{equation} \label{2.17}
         f(y) \geq \left( \prod_{i=1}^n y_i \right)^\frac{1}{n} = \left( \prod_{i=1}^n E_{m-1}(\hat{x}^i) \right)^{-\frac{1}{n}}.
     \end{equation}
     Combining \eqref{2.15} and \eqref{2.17}, we can obtain
     \begin{equation}
         f(x) \leq \left( \prod_{i=1}^n E_{m-1}(\hat{x}^i) \right)^{\frac{1}{n(m-1)}}, \quad \forall x\in\Gamma_m,
     \end{equation}
     as desired.

     Finally, we discuss the condition for equality in \eqref{2.10}. If equality holds, from the above proof, we conclude that equality must hold in inequality \eqref{AM-GM}. Therefore, the products of any $m$ components of $y$ are equal, which leads to $y_1=\cdots=y_n$, then we can obtain $\sigma_{m-1}(\hat{x}^1)=\cdots=\sigma_{m-1}(\hat{x}^n)$.
    
Let $\hat{x}^{ij}$ denotes the $(n-2)$-dimensional vector obtained by deleting both the $i$-th and $j$-th coordinates. By property \eqref{sum property}, for any $1\leq i<j\leq n$, we have 
\[
\begin{aligned}
\sigma_{m-1}(\hat{x}^i) &= \sigma_{m-1}(\hat{x}^{ij}) + x_j\sigma_{m-2}(\hat{x}^{ij}), \\
\sigma_{m-1}(\hat{x}^j) &= \sigma_{m-1}(\hat{x}^{ij}) + x_i\sigma_{m-2}(\hat{x}^{ij}).
\end{aligned}
\]
Subtracting gives
\begin{equation} \label{Subtract}
\sigma_{m-1}(\hat{x}^i) - \sigma_{m-1}(\hat{x}^j) = (x_j-x_i)\sigma_{m-2}(\hat{x}^{ij}).   
\end{equation}
Since $\sigma_{m-1}(\hat{x}^i) = \sigma_{m-1}(\hat{x}^j)$ for any $1\leq i<j\leq n$, and by Lemma \ref{Lemma3}, we have $\sigma_{m-2}(\hat{x}^{ij}) > 0$, so we deduce that $x_i=x_j$ from \eqref{Subtract}, for any $1\leq i<j\leq n$. This proves the lemma.
\end{proof}

\vspace{3mm}

\section{Proofs of the main results} \label{Section 3}
In this section, we give the proofs of the main results. We divide all the proofs into three subsections, in which we give the proofs of Theorems \ref{Th1.3}, \ref{Th1.4}, and \ref{Th1.5} together with the related corollaries, respectively.

\subsection{Proofs of Theorem \ref{Th1.3} and related corollaries}
In this subsection, we first prove Theorem \ref{Th1.3} based on Lemma \ref{Lemma2.2}. Then, we give the proofs of Corollaries \ref{Co1.1}, \ref{Co1.2} and \ref{Co1.3}, respectively.

\begin{proof} [Proof of Theorem \ref{Th1.3}]
First we assume $x\in\mathcal P_{k-1}$. Fix a $(k+1)$-tuple $I=\{i_1,...,i_{k+1}\} \subset \{1,\ldots,n\}$ with $i_1<\cdots<i_{k+1}$, and let $q = k+1$, then $q\geq3$. Define
\[
  x^{I} := (x_1^{I},\ldots,x_q^{I}) \in \mathbb{R}^q,
\]
where $x_l^{I}=x_{i_l}/s_I,\, \forall 1\leq l\leq q$, and $s_I:=\sum_{i\in I} x_i$. Since $\mathcal{P}_{k-1}\subset\mathcal{P}_{k+1}$, we have $s_I>0$. Obviously, we have $\sum_{j=1}^q x_j^I = 1$, and for any $1\leq j<l\leq q$,
\[
 1-x_j^{I}-x_l^{I} = \frac{\sum_{i\in I\setminus\{i_j,i_l\}} x_i}{s_I}>0,
\]
since $x\in\mathcal{P}_{k-1}$. Thus $x^{I}\in\mathcal D_q$, and Lemma \ref{Lemma2.2} can be applied to $x^I$.
For each $q-2\leq r\leq q$, we let
\begin{equation}\label{ocal-mean}
 M_{r}^I(x) := M_r(s_I x^I) = \frac{\Pi_r^{\frac{1}{C_q^r}}(s_I x^I)}{r}.
\end{equation}
When $x\in\mathcal{P}_{k-1}$, we have $M_r^I(x)>0$ for every $q-2\leq r\leq q$. Therefore, by direct computation,
\begin{equation} \label{Log1}
\begin{aligned}
 \log M_{q-2}^I(x) &= \frac{1}{C_q^2}\sum_{1\le i<j\le q}
 \log\left( \frac{s_I(1-x_i^{I}-x_j^{I})}{q-2} \right) \\
 &= \log\left(\frac{s_I}{q-2}\right) + \frac{1}{C_q^2}\sum_{1\le i<j\le q} \log(1-x_i^{I}-x_j^{I}),
\end{aligned}
\end{equation}

\begin{equation} \label{Log2}
\begin{aligned}
 \log M_{q-1}^I(x) &= \frac{1}{q} \sum_{i=1}^q \log\left(\frac{s_I(1-x_i^{I})}{q-1}\right) \\
 &= \log\left(\frac{s_I}{q-1}\right) + \frac1q\sum_{i=1}^q\log(1-x_i^{I}),
\end{aligned}
\end{equation}
and
\begin{equation} \label{Log3}
 \log M_{q}^I(x) = \log\left(\frac{s_I}{q}\right).
\end{equation}
Substituting \eqref{Log1}--\eqref{Log3}, with $k=q-1$, gives
\begin{equation} \label{local N2}
\begin{aligned}
 &\quad\ 2\log M_{k}^I(x) - \log M_{k-1}^I(x) - \log M_{k+1}^I(x) \\
 &=2\log\left(\frac{s_I}{q-1}\right) - \log\left(\frac{s_I}{q-2}\right) - \log\left(\frac{s_I}{q}\right) + \frac2q\sum_{i=1}^q\log(1-x_i^{I})
 -\frac{1}{C_q^2}\sum_{1\le i<j\le q}\log(1-x_i^{I}-x_j^{I}) \\
 &=\log\left(\frac{q(q-2)}{(q-1)^2}\right)
 +\frac{\mathcal L_q(x^{I})}{C_q^2} \\
 &=\frac{\mathcal L_q(x^{I})
 -\mathcal L_q(x^{q})}{C_q^2}\ge0,
\end{aligned}
\end{equation}
where Lemma \ref{Lemma2.2} was used, and the equality in \eqref{local N2} holds if and only if $x_i^{I}=1/q$ for all $1\leq i\leq q$, or equivalently, $x_{i_1}=\cdots=x_{i_q}=s_I/q$.

For any fixed $k-1\leq r\leq k+1$. Summing $\log M_r^I(x)$ over all $I$ with $I\subset\{1,...,n\}$ and $|I|=k+1$, where $|I|$ denotes the cardinality of $I$, we can obtain  
\begin{equation} \label{local-global}
\begin{aligned}
 \frac{1}{C_n^{k+1}} \sum_{\substack{I\subset\{1,\ldots,n\}\\|I|=k+1}}
 \log M_{r}^I(x) &= \frac1{C_{n}^{k+1}C_{k+1}^r} \sum_{\substack{I\subset\{1,\ldots,n\}\\|I|=k+1}} \sum_{\substack{S\subset I\\|S|=r}}
 \log\left(\frac1r\sum_{i\in S} x_i\right) \\
 &= \frac{C_{n-r}^{k+1-r}}{C_n^{k+1}C_{k+1}^r}
 \sum_{\substack{S\subset\{1,\ldots,n\}\\|S|=r}}
 \log\left(\frac1r\sum_{i\in S} x_i\right) \\
 &= \frac1{C_n^r} \sum_{\substack{S\subset\{1,\ldots,n\}\\|S|=r}}
 \log\left(\frac1r\sum_{i\in S} x_i\right) \\
 &=\log M_{r}(x).
\end{aligned}
\end{equation}
Thus, by \eqref{local N2}, we have
\begin{equation} \label{N2 global}
\begin{aligned}
 &\quad\ 2\log M_{k}(x) -\log M_{k-1}(x) -\log M_{k+1}(x) \\
 & =\frac{1}{C_n^{k+1}}
 \sum_{\substack{I\subset\{1,\ldots,n\}\\|I|=k+1}}
 \bigl(2\log M_{k}^I(x) -\log M_{k-1}^I(x) -\log M_{k+1}^I(x)\bigr) \ge0,
\end{aligned}
\end{equation}
which proves inequality \eqref{N2} on $\mathcal{P}_{k-1}$. 

For $x\in\overline{\mathcal P}_{k-1}$. Let $x_{\varepsilon}:=(x_1+\varepsilon,\ldots, x_n+\varepsilon)$, where $\varepsilon>0$ sufficiently small, then $x_\varepsilon\in\mathcal P_{k-1}$. Hence by previous result, we have  
\begin{equation} \label{N2-epsilon}
  M_{k}^2(x_\varepsilon) \ge M_{k-1}(x_{\varepsilon}) M_{k+1}(x_{\varepsilon}).
\end{equation}
Clearly, $M_{r}$ is continuous for any $k-1\leq r\leq k+1$, and $x_\varepsilon\to x$ as $\varepsilon\to0$, so let $\varepsilon\to 0$ in \eqref{N2-epsilon}, we can obtain the desired result.

Finally, we discuss when equality holds in inequality \eqref{N2}. If $x\in\mathcal{P}_{k-1}$, combining \eqref{local N2}, \eqref{N2 global} and the equality condition in \eqref{local N2}, we can easily obtain that equality holds in \eqref{N2} if and only if $x_1=\cdots=x_n>0$. If $x\in\partial\mathcal{P}_{k-1}$, then there exist at least $k-1$ components of $x$ whose sum is $0$, so $M_{k-1}(x)=0$. Equality is therefore equivalent to $M_{k}(x)=0$. When $M_k(x)=0$, there exists a $k$-tuple $I\subset\{1,...,n\}$ with $\sum_{i\in I}x_i=0$. For each $i\in I$, we have
\[
 -x_i = \sum_{j\in I}x_j-x_i = \sum_{j\in I\setminus\{i\}} x_j \ge 0.
\]
Thus, we can deduce that $x_i=0$ for any $i\in I$. Moreover, we have $x_i\geq0$ for any $i\notin I$. When $k=2$, this is obvious; for $k\geq3$ and $i\notin I$, we can choose a $(k-2)$-tuple $S\subset I$, then we have
\begin{equation*}
    x_i = x_i + \sum_{j\in S} x_j = \sum_{j\in S\cup\{i\}} x_j \ge 0,
\end{equation*}
as desired. Conversely, if $x$ contains $k$ zeros, then $M_k(x)=0$ clearly.
\end{proof}

\begin{remark}
We note that the inequality \eqref{N2} is a Newton-type inequality, but its range of validity differs from that of the classical Newton's inequality. The classical Newton's inequality holds on all of $\mathbb{R}^n$, mainly because of the real-rootedness property of polynomials. Specifically, for any $x=(x_1,...,x_n)\in\mathbb{R}^n$, the elementary symmetric means can arise as normalized coefficients of the degree $n$ polynomial $p(t)$ with roots $x_1,...,x_n$:
\[
   p(t) = \prod_{i=1}^n (t-x_i) = \sum_{k=0}^n (-1)^k C_n^k E_k(x) t^{n-k}.
\]
However, the Maclaurin means $M_k(x)$, $1\leq k\leq n$, clearly lacks such a structure, so we adopt a different approach and prove that inequality \eqref{N2} holds on $\overline{\mathcal{P}}_{k-1}$, but we cannot expect inequality \eqref{N2} holds on all of $\mathbb{R}^n$. In general, restricting to $\overline{\mathcal{P}}_{k-1}$ is natural and necessary, as it ensures that the Maclaurin means are all real-valued. In fact, the inequality \eqref{N2} can holds only on all of $\mathbb{R}^3$, see Remark \ref{remark 3.2}. However, for $n\geq 4$, even if the Maclaurin means are all real-valued, inequality \eqref{N2} generally no longer holds, see the counterexample in Remark \eqref{remark 3.3}. 
\end{remark}

\begin{remark} \label{remark 3.2}
For $n=3$ and $k=2$, the Newton-type inequality \eqref{N2} holds on all of $\mathbb{R}^3$. In this case, the inequality is equivalent to
\begin{equation} \label{N2 n=3}
27\bigl[(x+y)(x+z)(y+z)\bigr]^2
\ge64xyz(x+y+z)^3,\quad \forall x,y,z\in\mathbb{R}.
\end{equation}
In view of Theorem \ref{Th1.3}, we already know that inequality \eqref{N2 n=3} holds for $x,y,z\geq0$. Write
\[
p=x+y+z,\quad q=xy+xz+yz,\quad r=xyz,\quad S=(x+y)(x+z)(y+z)=pq-r.
\]
If $pr\leq0$, then inequality \eqref{N2 n=3} holds obviously. Thus, we suppose $pr>0$. Because replacing $(x,y,z)$ by $(-x,-y,-z)$ in inequality \eqref{N2 n=3} leaves both sides unchanged, it suffices to consider the case where $p>0$, $r>0$, and among $x,y,z$, two are negative and one is positive. Without loss of generality, we write $(x,y,z)=(a,-b,-c)$, with $a,b,c>0$. Since $p>0$, we have $a > b+c$, and
\[
q = bc-a(b+c) < bc-(b+c)^2 < 0.
\]
Thus $pq<0$ and $r>0$, so $(pq-r)^2>(pq)^2$. Moreover, by the classical Newton's inequality \eqref{N1}, we have $q^2\geq3pr$. Then we have
\begin{equation*}
S^2 > p^2q^2 \ge 3p^3r > \frac{64}{27} p^3r,
\end{equation*}
as desired.
\end{remark}

\begin{remark} \label{remark 3.3}
The Newton-type inequality \eqref{N2} generally does not hold on all of $\mathbb{R}^n$ for $n\geq4$ and $2\leq k \leq n-1$, even if the Maclaurin means are all real-valued. To see this, we will show that for every $n\ge4$ and $2\le k\le n-1$, there exists some $x^\ast\in\mathbb{R}^n$ such that
\begin{equation*}
M_{k}(x^\ast) = 0,\quad M_{k-1}(x^\ast)>0,\quad M_{k+1}(x^\ast)>0.
\end{equation*}
Thus, $x^\ast$ does not satisfy the inequality \eqref{N2}. 

For $k=2$, we can take
\begin{equation*}
x^\ast = (1,1,-1,-1,\underbrace{3,...,3}_{n-4}).
\end{equation*}
We note that the above $x^\ast\in\overline{\mathcal{P}}_4$, but $x^\ast\notin\overline{\mathcal{P}}_1$. Clearly, by simple calculation, we have $M_2(x^\ast)=0$, $M_1(x^\ast)>0$ and $M_3(x^\ast)>0$. 

For $k\ge3$, let $k_0=\frac{k-2}{2}>0$, and we can take
\begin{equation*}
x^\ast = (k_0,k_0,\underbrace{-1,...,-1}_{k-2},
\underbrace{k,...,k}_{n-k}).
\end{equation*}
We note that the above $x^\ast\in\overline{\mathcal{P}}_k$, but $x^\ast\notin\overline{\mathcal{P}}_{k-1}$. By direct computation, we can easily get $M_k(x^\ast)=0$, $M_{k-1}(x^\ast)>0$ and $M_{k+1}(x^\ast)>0$.

\end{remark}

Next, we prove Corollary \ref{Co1.1} by applying Theorem \ref{Th1.3}.
\begin{proof} [Proof of Corollary \ref{Co1.1}] 

We only consider $x\in\mathcal P_{p}$. Otherwise, as in the proof of Theorem \ref{Th1.3}, we can use a limiting argument. In fact, the Newton-type inequality \eqref{N2} implies the concavity of $\log M_k(x)$ with respect to $k$. Thus by the basic property of concave function, for any $1\leq p<r<s\leq n$ and $x\in\mathcal{P}_p$, we have
\[
 \frac{\log M_{r}(x)-\log M_{p}(x)}{r-p}
 \ge
 \frac{\log M_{s}(x)-\log M_{r}(x)}{s-r},
\]
which leads to
\[
 (s-p)\log M_{r}(x) \ge (s-r)\log M_{p}(x) + (r-p)\log M_{s}(x).
\]
Therefore, we can obtain inequality \eqref{Interpolation} by a simple transformation.
\end{proof}

Below, we give the proof of Corollary \ref{Co1.2} based on Theorem \ref{Th1.3}.

\begin{proof} [Proof of Corollary \ref{Co1.2}]
  In fact, inequality \eqref{N2} is stronger than inequality \eqref{M2}, so inequality \eqref{M2} can be derived from Newton-type inequality \eqref{N2} easily similar to the classical case. To prove inequality \eqref{M2}, it suffices to prove that for any $1\leq k\leq n-1$ and $x\in\overline{\mathcal{P}}_{k}$, the following inequality holds:
  \begin{equation} \label{M2 k}
      M_{k+1}(x) \geq M_k(x).
  \end{equation}
To see this, we can use backward induction on $k$. Firstly, we prove \eqref{M2 k} holds for $k = n-1$ on $\overline{\mathcal{P}}_{n-1}$. As before, we only let $x=(x_1,...,x_n)\in\mathcal{P}_{n-1}$, then by AM-GM inequality, we have 
\begin{equation}
    \Pi^{\frac{1}{n}}_{n-1}(x) = \left( \prod_{i=1}^n \left( \Pi_n(x)-x_i \right) \right)^\frac{1}{n} \leq \frac{\sum_{i=1}^n \left( \Pi_n(x)-x_i \right) }{n} = \frac{n-1}{n} \Pi_n(x),
\end{equation}
as desired.

Next, we assume that \eqref{M2 k} holds for $2\leq k\leq n-1$ on $\overline{\mathcal{P}}_{k}$, then we need to prove $ M_k(x) \geq M_{k-1}(x)$ on $\overline{\mathcal{P}}_{k-1}$. Without loss of generality, we only suppose $x=(x_1,...,x_n)\in\mathcal{P}_{k-1}$, then we have $x\in\mathcal{P}_{k}$, thus by the inequalities \eqref{N2} and \eqref{M2 k}, we obtain
\begin{equation}
    M_k^2(x) \geq M_{k-1}(x) M_{k+1}(x) \geq M_{k-1}(x) M_k(x),
\end{equation}
which implies that $M_k(x) \geq M_{k-1}(x)$ on $\mathcal{P}_{k-1}$, since $M_k(x)>0$. Thus, we complete the proof of inequality \eqref{M2}.

Finally, it is not difficult to obtain the equality condition of inequality \eqref{M2}. If $x\in\mathcal{P}_j$, then by the above proof, the equality condition follows directly from the equality condition of Newton-type inequality \eqref{N2}. If $x\in\partial\mathcal{P}_j$, then the discussion is similar to that of the Newton-type inequality \eqref{N2}, see the last part of the proof of Theorem \ref{Th1.3}.
\end{proof}

\begin{remark}
    In fact, in \cite{carlson1971mixed}, inequality \eqref{M2} has already been proved on $\overline{K^+}$, and we find that the proof in \cite{carlson1971mixed} is sufficient to derive that inequality \eqref{M2} also holds on $\overline{\mathcal{P}}_j$. Nevertheless, we still give a completely different proof from that in \cite{carlson1971mixed}, which is based on Newton-type inequality \eqref{N2} and backward induction.
\end{remark}

\begin{remark}
    Furthermore, we find that inequality \eqref{M2} can also be obtained directly from a class of Muirhead-type inequality. Given $\alpha\in \mathcal{H}(\beta)$, where $\alpha=(\alpha_1,...,\alpha_n)\in\mathbb{R}^n$, $\beta=(\beta_1,...,\beta_n)\in\mathbb{R}^n$, and $\mathcal{H}(\beta)$ denotes the convex hull of all vectors obtained by permuting the coordinates of $\beta=(\beta_1,...,\beta_n)$. The classical Muirhead's inequality asserts that for any $x=(x_1,...,x_n)\in K^+$, one has
\begin{equation} \label{M-H1}
    \sum_{\sigma\in S_n} x_{\sigma(1)}^{\alpha_1}\cdots x_{\sigma(n)}^{\alpha_n} \leq \sum_{\sigma\in S_n} x_{\sigma(1)}^{\beta_1}\cdots x_{\sigma(n)}^{\beta_n},
\end{equation}
    where $S_n$ denotes the symmetric group on $n$ letters. For the proof of Muirhead's inequality \eqref{M-H1}, one can refer to \cite{steele2004cauchy}. 
    
    In fact, we can easily derive a similar Muirhead-type inequality, which can be viewed as the counterpart of \eqref{M-H1}. Let $\alpha\in \mathcal{H}(\beta)$, where $\alpha=(\alpha_1,...,\alpha_n)\in\mathbb{R}^n$ and $\beta=(\beta_1,...,\beta_n)\in\mathbb{R}^n$, then for any $x=(x_1,...,x_n)\in\mathbb{R}^n$ satisfying $\sum_{i=1}^n\beta_{\tau(i)}x_{\sigma(i)}\geq0$ for any $\sigma,\tau\in S_n$, the following Muirhead-type inequality holds:
    \begin{equation} \label{M-H2}
        \prod_{\sigma\in S_n} \left(\alpha_1 x_{\sigma(1)}+\cdots+ \alpha_n x_{\sigma(n)}\right) \geq \prod_{\sigma\in S_n} \left(\beta_1 x_{\sigma(1)}+\cdots+ \beta_n x_{\sigma(n)}\right).
    \end{equation}
The proof of inequality \eqref{M-H2} is similar to that of inequality \eqref{M-H1}, so we omit it here, the interested reader may refer to \cite{steele2004cauchy}, where the proof of inequality \eqref{M-H2} is not difficult to give.

As mentioned earlier, the Muirhead-type inequality \eqref{M-H2} directly implies inequality \eqref{M2}. To see this, we take 
\begin{equation*}
   \alpha=(\underbrace{j,...,j}_{l},\underbrace{0,...,0}_{n-l}) \quad \text{and} \quad \beta=(\underbrace{l,...,l}_{j},\underbrace{0,...,0}_{n-j}),
\end{equation*}
where $1\leq j<l\leq n$. Clearly, we have $\alpha\prec\beta$, which denotes $\alpha$ is majorized by $\beta$. For the detailed definition of majorization, one can see \cite{steele2004cauchy}. In fact, for any $n$-tuples $\alpha$ and $\beta$, we have 
$$
\alpha\prec\beta \iff \alpha\in \mathcal{H}(\beta),
$$
also see \cite{steele2004cauchy} for the proof. Thus, for any $x=(x_1,...,x_n)\in\overline{\mathcal{P}}_j$, from the Muirhead-type inequality \eqref{M-H2}, we can obtain
\begin{equation}
    j^{n!} \Pi_{l}^{l!(n-l)!}(x) \geq l^{n!} \Pi_{j}^{j!(n-j)!}(x),
\end{equation}
then taking the $n!$-th root of both sides yields the desired inequality \eqref{M2}.
\end{remark}

\begin{remark}
    The classical Muirhead's inequality \eqref{M-H1} does not directly imply the classical Maclaurin's inequality \eqref{M1}, while the dual Muirhead's inequality \eqref{M-H2} directly implies the Maclaurin-type inequality \eqref{M2}. This difference seems to be interesting.
\end{remark}

Finally, we present the proof of Corollary \ref{Co1.3} based on Corollary \ref{Co1.2}.

\begin{proof} [Proof of Corollary \ref{Co1.3}]
    In fact, the proof of this result is direct due to Corollary \ref{Co1.2}. Since $\sum_{i=1}^n r_i = m$, then we let $y=(y_1,...,y_m)$, which is comprised of $r_i$ copies of $x_i$, for $1\leq i\leq n$. Then we note that the following combinatorial identity holds:
    \begin{equation}
      \Pi_l(y) = \prod_{1\leq i_1<\cdots<i_{l}\leq m}(y_{i_1}+\cdots+y_{i_l}) = \prod_{(l_1,...,l_n)\in\mathcal{S}_{n,l}} \left( \sum_{i=1}^n l_ix_i \right)^{\prod_{i=1}^n C_{r_i}^{l_i}}.
    \end{equation}
    Thus, by Corollary \ref{Co1.2}, we obtain the desired inequality \eqref{wighted M2}, and the equality holds if and only if $y_1=\cdots=y_m > 0$ or $y$ contains at least $l$ zeros, the former case is equivalent to $x_1=\cdots=x_n > 0$; and the latter case is equivalent to $\sum_{i=1}^n r_i\text{\large$\chi$}_{\{0\}}(x_i)\geq l$.
\end{proof}

\begin{remark}
  Unlike Corollary \ref{Co1.2}, we require $x=(x_1,...,x_n)\in\overline{K^+}$ in Corollary \ref{Co1.3}. This is because, in general, $x=(x_1,...,x_n)\in\overline{\mathcal{P}}_j$ does not ensure $y=(y_1,...,y_m)\in\overline{\mathcal{P}}_j$ for $j>1$, and therefore Corollary \ref{Co1.2} does not apply. For instance, let $n=3$, $(w_1, w_2, w_3)=(\frac{1}{2}, \frac{1}{3}, \frac{1}{6})$, $m=6$, i.e., $r=(3, 2, 1)$, for $j=2$ and $x=(-1, 1, 2)\in\overline{\mathcal{P}}_2$, then $y=(-1, -1, -1, 1, 1, 2)\notin\overline{\mathcal{P}}_2$.
\end{remark}

\begin{remark}
The inequalities \eqref{wighted M2} are another interpolating sequence of inequalities between the weighted arithmetic and geometric means, since
$$
P_1^{\frac{1}{C_{m}^1}}(x,r) = \prod_{i=1}^n {x_i}^{w_i} \quad \text{and} \quad \frac{P_m^{\frac{1}{C_{m}^m}}(x,r)}{m} = \sum_{i=1}^n w_ix_i.
$$
Moreover, if we let $w_i=\frac{1}{n}$ for any $1\leq i\leq n$ and $m=n$, then inequality \eqref{wighted M2} reduces to inequality \eqref{M2}. Thus, Corollary \ref{Co1.3} can be regarded as the counterpart of those in \cite{favaro2018general}.
\end{remark}

\subsection{Proofs of Theorem \ref{Th1.4} and related corollaries} 
In this subsection, we mainly give the proof of Theorem \ref{Th1.4}. The proofs of Corollary \ref{Co1.4} and Corollary \ref{Co1.6} are obvious by combining the classical Maclaurin's inequality \eqref{M1}, Corollary \ref{Co1.2} and Theorem \ref{Th1.4}, so we omit them here.

\begin{proof} [Proof of Theorem \ref{Th1.4}]
 For $k=1$, the result is trivial. Let $2\leq k\leq n$, we shall prove inequality \eqref{M3} by induction on the number of variables $n\geq k$, where $k\geq2$ is fixed. As before, we only suppose $x\in\Gamma_{n-k+1}$, and we let $m:=n-k+1$. When $n=k$, we have $m=1$, and 
 \[
   M_k(x) = \frac{x_1+\cdots+x_k}{k} = E_1(x),
 \]
 so inequality \eqref{M3} holds for $n=k$. Let $n>k$, we have $2\leq m\leq n-1$, and we assume that the inequality \eqref{M3} holds for $n-1$ variables. If $x=(x_1,...,x_n)\in \Gamma_m$, by Lemma \ref{Lemma3}, we have $\hat{x}^i\in\Gamma_{m-1}\subset\mathcal{P}_k$ for any $1\leq i\leq n$. Since $(n-1)-k+1=m-1$, the induction hypothesis yields
 \begin{equation} \label{induction}
      M_{k}(\hat{x}^i)\ge E_{m-1}^\frac{1}{m-1}(\hat{x}^i), \quad \forall 1\leq i\leq n.
 \end{equation}
 Moreover, analogous to \eqref{local-global}, we can obtain
\begin{equation} \label{identity-M_k}
M_{k}(x)=\left(\prod_{i=1}^n M_{k}(\hat{x}^i)\right)^\frac{1}{n}.
\end{equation}
Then, combining \eqref{induction}, \eqref{identity-M_k}, and by Lemma \ref{Lemma4}, we have
\begin{equation}
M_{k}(x) = \left(\prod_{i=1}^n M_{k}(\hat{x}^i)\right)^\frac{1}{n} \ge \left( \prod_{i=1}^n E_{m-1}^\frac{1}{m-1}(\hat{x}^i) \right)^\frac{1}{n} \ge E_{m}^\frac{1}{m}(x).
\end{equation}
Thus, the inequality \eqref{M3} also holds for $n$ variables, and we complete the proof of inequality \eqref{M3}.

Finally, we discuss the equality condition of inequality \eqref{M3}. When $k=1$ or $k=n$, both sides of \eqref{M3} coincide. Henceforth, we assume that \(2\le k\le n-1\), so that \(2\le m\le n-1\). If $x\in\Gamma_m$, and the equality holds in \eqref{M3}. From the above proof, since we employed inequality \eqref{2.10}, the equality also holds in \eqref{2.10}, then by Lemma \ref{Lemma4}, we can obtain $x_1=\cdots=x_n > 0$. If $x\in\partial\Gamma_m$, then there exists some $1\leq j\leq m$, such that $\sigma_{j}(x)=0$. If $j\neq m$, then by the Maclaurin's inequality \eqref{M1}, we always have $\sigma_m(x)=0$. Equality is therefore equivalent to $M_{k}(x)=0$. When $M_k(x)=0$, there exists a $k$-tuple $I\subset\{1,...,n\}$ with $\sum_{i\in I}x_i=0$. Since $k\leq n-1$, we can choose an index $1\leq l\leq n$ with $l\notin I$. Let $y$ denote the $(k+1)$-dimensional vector formed by keeping only the coordinates of $x$ corresponding to the indices in $I\cup\{l\}$. If $m=2$, then $y = x$, and $\sigma_2(y)=0$. If $2<m\leq n-1$, since $n-(k+1)=m-2$, by Lemma \ref{Lemma3} and a limiting argument, we can obtain $\sigma_2(y)\geq 0$. On the other hand, by direct calculation, we have
\begin{equation}
\begin{aligned}
    0\leq 2\sigma_2(y) &= 2\sum_{\substack{i,j\in I \\ i<j}} x_ix_j + 2x_l \sum_{i\in I} x_i \\
    &= 2\sum_{\substack{i,j\in I \\ i<j}} x_ix_j \\
    &= \left(\sum_{i\in I}x_i\right)^2 - \sum_{i\in I} x_i^2 \\
    &= - \sum_{i\in I} x_i^2 \leq 0.
\end{aligned}
\end{equation}
Thus, we have $\sum_{i\in I}x_i^2=0$, which leads to $x_i=0$ for any $i\in I$. Moreover, we have $x_i\geq 0$ for any $i\notin I$, this is because $\overline{\Gamma}_m\subset \overline{\mathcal{P}}_k$, then the result follows analogous to the discussion of the Newton-type inequality \eqref{N2}, see the last part of the proof of Theorem \ref{Th1.3}. Conversely, if $x$ contains $k$ zeros, then $M_k(x)=0$ clearly.
\end{proof}

\begin{remark}
    For $n=3$ and $k=2$, there is another elegant proof of the inequality \eqref{M3} based on a trigonometric transformation, which can be found in \cite{sedrakyan2018algebraic}, see Problem 6.8. Although \cite{sedrakyan2018algebraic} proves the inequality only on $\overline{K^+}$, the method used there is sufficient to prove it on $\overline{\Gamma}_2$. To see this, let $x=(x_1,x_2,x_3)\in \Gamma_2$, without loss of generality, we can assume that $\sigma_2(x)=1$, since inequality \eqref{M3} is homogeneous. Since $\Gamma_2 \subset \mathcal{P}_2$ in $\mathbb{R}^3$ (see \cite{li2020inclusion}), we conclude that $x$ contains at least two positive coordinates. Because inequality \eqref{M3} is symmetric, without loss of generality, we may assume that $x_1, x_2>0$. Let $x_1 = \tan\alpha$, $x_2=\tan \beta$, $x_3=\tan\gamma$, where $\alpha,\beta\in(0,\frac{\pi}{2})$ and $\gamma\in(-\frac{\pi}{2},\frac{\pi}{2})$. From $\sigma_2(x)=1$, it follows that 
    $$
    \tan(\alpha+\beta) = \frac{\tan\alpha + \tan\beta}{1-\tan\alpha \tan\beta} = \frac{1}{\tan\gamma} = \tan(\frac{\pi}{2}-\gamma).
    $$
    Thus, we obtain $\alpha+\beta+\gamma = \frac{\pi}{2} + k\pi$, where $k\in\mathbb{Z}$. Since $\alpha+\beta+\gamma\in(-\frac{\pi}{2},\frac{3\pi}{2})$, we obtain $\alpha+\beta+\gamma=\frac{\pi}{2}$, which is the key point. Then the rest of the proof is the same as that in \cite{sedrakyan2018algebraic}, and is therefore omitted. In fact, the above elegant relation exists only on $\mathbb{R}^3$, thus the proof is available only on $\mathbb{R}^3$. This is also why, in fully nonlinear partial differential equations, the $2$-Hessian equation $\sigma_2(\lambda)=1$ is equivalent to the special Lagrangian equation $\sum_{i=1}^3 \arctan \lambda_i = \frac{\pi}{2}$ on $\mathbb{R}^3$, where $\lambda=(\lambda_1,\lambda_2,\lambda_3)$ denotes the eigenvalues of the Hessian $D^2u$, whereas no such analogous equivalence exists in higher-dimensional spaces. 
\end{remark}

\subsection{Proofs of Theorem \ref{Th1.5} and related corollaries}
In this subsection, we give the proof of Theorem \ref{Th1.5} and Corollary \ref{Co1.5}. The proof of Corollary \ref{Co1.7} is obvious by combining the classical Maclaurin's inequality \eqref{M1}, Corollary \ref{Co1.2} and Theorems \ref{Th1.4}, \ref{Th1.5}, so we omit it here.

\begin{proof} [Proof of Theorem \ref{Th1.5}]
    We divide the proof into two parts. 

    \vspace{2mm}
    
    \textbf{Case 1.} $n=kr$, $1<k,r<n$. In this case, we select $r$ numbers from the $kr$ numbers $x_1,...,x_{kr}$ and sum them, then select another $r$ numbers from the remaining $(k-1)r$ numbers and sum them, and continue this process. Ultimately, we obtain $k$ sums. Multiplying these $k$ sums together yields a total of $m := \prod_{i=0}^{k-1} C_{(k-i)r}^r$ such products, denoted by $y_1,...,y_m$.
    Since $x\in\overline{\mathcal{P}}_r$, we have $y_i\geq0$ for any $1\leq i\leq m$.
    
    By direct calculation, we have 
    \begin{equation} \label{S}
        y_1+ \cdots + y_m = k! \prod_{i=0}^{k-1} C_{(k-i)(r-1)}^{r-1} \sigma_k(x) \geq 0, 
    \end{equation}
    since for any given $x_{i_1},...,x_{i_k}$, where $1\leq i_1< \cdots <i_k\leq kr$, there are $k! \prod_{i=0}^{k-1} C_{(k-i)(r-1)}^{r-1}$ occurrences of $x_{i_1}\cdots x_{i_k}$ in the expansion of $(y_1+\cdots+y_m)$.

    Furthermore, we have
    \begin{equation} \label{P}
        y_1\cdots y_m = \Pi_r^{\frac{km}{C_{kr}^r}}(x),
    \end{equation}
    since for any given $x_{i_1},...,x_{i_r}$, where $1\leq i_1<\cdots <i_r\leq kr$, $y_1,...,y_m$ together contain $C_k^1\cdot\prod_{i=1}^{k-1} C_{(k-i)r}^r$ copies of $(x_{i_1} +\cdots+ x_{i_r})$. 
    
    Thus, by the AM-GM inequality, we can obtain
    \begin{equation}
      \frac{ k! \prod_{i=0}^{k-1} C_{(k-i)(r-1)}^{r-1} }{m} \sigma_k(x) = \frac{y_1 + \cdots +y_m}{m} \geq \sqrt[m]{y_1\cdots y_m} = \Pi_r^{\frac{k}{C_{kr}^r}}(x).
    \end{equation}
    We can verify that 
    \begin{equation} \label{2.45}
        m \cdot r^k = k! \prod_{i=0}^{k-1} C_{(k-i)(r-1)}^{r-1} \cdot C_{kr}^k.
    \end{equation}
    In fact, \eqref{2.45} is the number of terms in the expansion of $(y_1+\cdots+y_m)$. Then we have
    \begin{equation}
        r^k \frac{\sigma_k(x)}{C_{kr}^k} \geq \Pi_r^{\frac{k}{C_{kr}^r}}(x),
    \end{equation}
    which implies inequality \eqref{M4}, as desired.

    Finally, we discuss the equality condition of inequality \eqref{M4}. For convenience, we define $s_I(x):=\sum_{i\in I}x_i$ for any $x\in\mathbb{R}^n$ and given $I\subset\{1,...,n\}$. With this notation, from the above proof, we can obtain that equality holds in inequality \eqref{M4} if and only if all products of the form $\prod_{i=1}^k s_{I_i}(x)$
    are equal, for any $k$ pairwise disjoint $r$-tuples $I_1,...,I_k\subset\{1,...,n\}$. Below, based on this, we present the more refined equality condition stated in Theorem \ref{Th1.5}.

    If $x\in\mathcal{P}_r$, we show that the equality in inequality \eqref{M4} holds if and only if, after a coordinate permutation,
    \begin{equation} \label{ac}
    x=(a,\underbrace{c,\dots,c}_{kr-1}),\quad c>0,\ a>-(r-1)c.
    \end{equation}
   If $x$ has the above form, then the equality is trivial. Conversely, we suppose the equality holds. If all coordinates of $x$ are equal, the conclusion is also trivial. Otherwise, we can take \(a:=x_i\) and \(b:=x_j\) for some $i\ne j$, where \(a\ne b\). Among the remaining \(n-2\) indices $\{1,...,n\}\setminus\{i,j\}$, we can select disjoint $(r-1)$-tuples \(I,J\subset\{1,...,n\}\setminus\{i,j\}\) . Since $(n-2)-2(r-1)=(k-2)r$. Thus, if $k\geq3$, we can choose $k-2$ pairwise disjoint $r$-tuples $L_1,...,L_{k-2}\subset\{1,...,n\}\setminus(I\cup J\cup\{i,j\})$, and we have
  \[
   s_{I\cup\{i\}}(x) s_{J\cup\{j\}}(x) \prod_{i=1}^{k-2} s_{L_i}(x) = s_{I\cup\{j\}}(x) s_{J\cup\{i\}}(x) \prod_{i=1}^{k-2} s_{L_i}(x).
  \] 
   After cancelling the common factors, we obtain
 \begin{equation} \label{ab}
  (a+\alpha)(b+\beta)=(b+\alpha)(a+\beta),  
\end{equation}
where $\alpha:=s_I(x)$ and $\beta:=s_J(x)$. If $k=2$, then we can obtain \eqref{ab} directly. Therefore, from \eqref{ab}, we obtain
\[
(a-b)(\alpha-\beta)=0.
\]
Since \(a\ne b\), we get 
\begin{equation} \label{alpha = beta}
    \alpha=\beta.
\end{equation}
Next, we claim that all coordinates of $x$ except \(x_i\) and $x_j$ are equal, with common value denoted by \(c\). If $k=2$, then the total number of remaining indices is $2(r-1)$, and by \eqref{alpha = beta}, we can obtain 
\[
s_I(x)=\frac{1}{2}\sum_{\substack{l=1 \\ l\neq i,j}}^n x_l,  
\] 
for every $(r-1)$-tuple $I\subset\{1,...,n\}\setminus\{i,j\}$ . Then we can easily obtain the desired result. If $k\ge3$, then from \eqref{alpha = beta}, we can easily obtain that $s_I(x)=s_J(x)$ for any $(r-1)$-tuples $I,J\subset\{1,...,n\}\setminus\{i,j\}$ differing in one index, i.e., \(|I\cap J|=r-2\), which also yields the desired result. This proves the claim. Thus, by the equality condition, we have
\[
    (rc)^{k-2} \left( a+(r-1)c \right) \left( b+(r-1)c \right)=(rc)^{k-1} \left( a+b+(r-2)c \right).
\]
Since $x\in\mathcal{P}_r$, we can eliminate $(rc)^{k-2}$, and by rearrangement, we get
\[
(a-c)(b-c)=0,
\]
which implies that $a=c$ or $b=c$. Without loss of generality, we may suppose $b=c$, and after a coordinate permutation, $x$ has the form \eqref{ac}, as desired.

If $x\in\partial\mathcal{P}_r$, then there exists a $r$-tuple $I\subset\{1,...,n\}$ such that $s_I(x)=0$, so \(\Pi_r(x)=0\) and equality holds in inequality \eqref{M4} if and only if $\prod_{i=1}^k s_{I_i}(x)=0$ for any $k$ pairwise disjoint $r$-tuples $I_1,...,I_k\subset\{1,...,n\}$.

If $x$ contains no negative coordinates, we will show that equality holds in inequality \eqref{M4} if and only if $x$ contains at least $n-k+1$ zeros, or equivalently,
\[
d:=|\{1\leq i\leq n: x_i>0\}|\le k-1.
\]
If \(d\le k-1\), then in any \(k\) pairwise disjoint $r$-tuples, at least one $r$-tuple $I\subset\{1,...,n\}$ contains no indices of positive coordinates, so $s_I(x) = 0$. Thus, the equality holds in inequality \eqref{M4}. Conversely, if \(d\ge k\), we can choose $k$ pairwise disjoint $r$-tuples $I_1,...,I_k\subset\{1,...,n\}$ such that each $r$-tuple $I_i$ contains an index of positive coordinate, where $1\leq i\leq k$. Thus, we have $\prod_{i=1}^k s_{I_i}(x) > 0$, which contradicts the equality condition.

Next, without loss of generality, we may rearrange the coordinates of $x$ as  
\begin{equation} \label{increasing}
    x_1\le x_2\le\cdots\le x_n.
\end{equation}
Since $x\in\partial\mathcal{P}_r$, we have 
\begin{equation} \label{sum = 0}
   \sum_{i=1}^r x_i = 0. 
\end{equation}
If $x$ contains negative coordinates, let $c:=x_r$, and we define
\[
L:=\{1\leq i\le n: x_i<c\}, \quad H:=\{1\leq i\leq n: x_i=c\}, \quad R:=\{1\leq i\leq n: x_i>c\}.
\]
Write \(\ell:=|L|\). By \eqref{increasing} and \eqref{sum = 0}, we can infer that $x_1<0$, $c>0$, \(1\le \ell\le r-1\), and
\begin{equation} \label{E sum = 0}
    \sum_{i\in L}x_i+(r-\ell)c=0.
\end{equation}
With the above notation, it is obvious that for any $r$-tuple $I\subset\{1,...,n\}$,
\begin{equation} \label{iff}
    s_I(x)=0 \iff L\subset I \ \ \text{and} \ \ I\setminus L\subset H.
\end{equation}
In other words, the $r$-tuple $I\subset\{1,...,n\}$ such that $s_I(x) = 0$ are exactly those containing all indices of coordinates less than \(c\), together with \(r-\ell\) additional indices of coordinates chosen from those equal to \(c\). In this case, we will prove that the equality holds in \eqref{M4} if and only if after a coordinate permutation,
\begin{equation} \label{c}
    x=c(-(r-1),\underbrace{1,\dots,1}_{kr-1}),\quad c>0.
\end{equation}
To see this, we need to prove that $\ell = 1$ and $R=\emptyset$ for equality. If \(\ell\ge 2\), we can take $k$ pairwise disjoint $r$-tuples $I_i\subset\{1,...,n\}$, $1\leq i\leq k$, such that two of them each contain at least one index in $L$. Then no $r$-tuple contains all the indices in \(L\), so by \eqref{iff}, we have $s_{I_i}(x)>0$ for any $1\leq i\leq k$, which is impossible for equality. 
If \(\ell=1\) but \(R\ne\emptyset\), we can take $k$ pairwise disjoint $r$-tuples $I_i\subset\{1,...,n\}$, $1\leq i\leq k$, such that one of them contains the unique index in $L$ and one index in $R$. Thus by \eqref{iff}, we also have $s_{I_i}(x)>0$ for any $1\leq i\leq k$, again equality is impossible. Therefore, when equality holds in inequality \eqref{M4}, except for the coordinate $x_1$, all the remaining coordinates of $x$ are equal to $c>0$. By \eqref{E sum = 0}, the exceptional coordinate $x_1=-(r-1)c$, which proves that $x$ has the form \eqref{c}. Conversely, if $x$ has the form \eqref{c}, then the equality holds clearly.

\vspace{2mm}

\textbf{Case 2.} $n>kr$, $1<k,r<n$. In this case, we can select any $(kr)$-tuple $I = \{i_1,...,i_{kr}\} \subset\{1,...,n\}$ with $i_1<\cdots<i_{kr}$. Let 
\begin{equation*}
    x^I := (x_1^I,...,x_{kr}^I)\in\mathbb{R}^{kr},
\end{equation*}
where $x_l^I =  x_{i_l}$, $\forall 1\leq l\leq kr$, then we have $x^I\in\overline{\mathcal{P}}_r$. Thus by the result of the first case, we have
    \begin{equation} \label{n=kr local}
        \frac{\sigma_k(x^I)}{C_{kr}^k} \geq \frac{\Pi_r^{\frac{k}{C_{kr}^r}}(x^I)}{r^k}.
    \end{equation}
    Summing \eqref{n=kr local} over all $(kr)$-tuples $I \subset\{1,...,n\}$, we obtain
    \begin{equation} \label{2.56}
       \frac{1}{C_{kr}^k} \sum_{\substack{I\subset\{1,\ldots,n\}\\|I|=kr}} \sigma_k(x^I) \geq \frac{1}{r^k} \sum_{\substack{I\subset\{1,\ldots,n\}\\|I|=kr}} \Pi_r^{\frac{k}{C_{kr}^r}}(x^I).
    \end{equation}
   Since $C_{n}^{kr}\cdot C_{kr}^k = C_{n}^k \cdot C_{n-k}^{kr-k}$, we have 
    \begin{equation} \label{2.57}
        \sum_{\substack{I\subset\{1,\ldots,n\}\\|I|=kr}} \sigma_k(x^I) = C_{n-k}^{kr-k} \sigma_k(x) \geq 0,
    \end{equation}
    and by the AM-GM inequality, we have
    \begin{equation} \label{2.58}
        \sum_{\substack{I\subset\{1,\ldots,n\}\\|I|=kr}} \Pi_r^{\frac{k}{C_{kr}^r}}(x^I) \geq C_n^{kr} \left( \prod_{\substack{I\subset\{1,\ldots,n\}\\|I|=kr}} \Pi_r^{\frac{k}{C_{kr}^r}}(x^I) \right)^\frac{1}{C_n^{kr}} = C_n^{kr} \Pi_r^\frac{kC_{n-r}^{kr-r}}{C_{kr}^r C_n^{kr}}(x) = C_n^{kr} \Pi_r^\frac{k}{C_n^r}(x).
    \end{equation}
    Combining \eqref{2.56}, \eqref{2.57} and \eqref{2.58}, we obtain
    \begin{equation}
      \frac{\sigma_k(x)}{C_n^k} = \frac{C_{n-k}^{kr-k}}{C_{kr}^k C_n^{kr}} \sigma_k(x) \geq \frac{\Pi_r^\frac{k}{C_n^r}(x)}{r^k},
    \end{equation}
    as desired. 
    
     Finally, we discuss the equality condition of inequality \eqref{M4}. From the above proof, analogous to the first case, we can also obtain that equality holds in inequality \eqref{M4} if and only if all products of the form $\prod_{i=1}^k s_{I_i}(x)$ are equal, for any $k$ pairwise disjoint $r$-tuples $I_1,...,I_k\subset\{1,...,n\}$. Below, based on this, we present the more refined equality condition stated in Theorem \ref{Th1.5}.
    
If $x\in\mathcal{P}_r$, then \(s_I(x)>0\) for any $r$-tuple $I\subset\{1,...,n\}$. Suppose equality holds, we choose any $r$-tuples \(I,J\subset\{1,...,n\}\) differing in one index, i.e., \(|I\cap J|=r-1\), so \(|I\cup J|=r+1\). Since \(n\ge kr+1\), we have
\[     
|\{1,...,n\}\setminus(I\cup J)|=n-r-1\ge (k-1)r.
\]
Hence, we can choose $k-1$ pairwise disjoint $r$-tuples $L_1,...,L_{k-1}\subset\{1,...,n\}\setminus(I\cup J)$, and by the equality condition, we have
\[
s_I(x) \prod_{i=1}^{k-1} s_{L_i}(x) = s_J(x) \prod_{i=1}^{k-1} s_{L_i}(x).
\]
Since \(s_{L_i}(x)>0\) for any $1\leq i\leq k-1$, we may cancel the common factors to yield
\[
s_I(x)=s_J(x).
\]
Thus, we can easily obtain $x_i=x_j$ for any $i\neq j$.

 If \(x\in\partial\mathcal{P}_r\), the discussion of the case where $x$ contains no negative coordinates is the same as that of the first case when $n=kr$, so in what follows we only discuss the case where $x$ contains negative coordinates. In fact, if $x$ contains negative coordinates, we can show that inequality \eqref{M4} is strict when $n>kr$. To see this, we first take an index \(i_0\in L\). Since \(n-1\ge kr\), one can choose \(k\) pairwise disjoint $r$-tuples $I_1,...,I_k \subset \{1,...,n\}\setminus\{i_0\}$. Each $I_i$ does not contain \(i_0\in L\), where $1\leq i\leq k$, so by \eqref{iff}, we have $s_{I_i}(x) > 0$ for any $1\leq i\leq k$. Thus we have $\prod_{i=1}^k s_{I_i}(x) > 0$, which is impossible for equality in inequality \eqref{M4}.

Thus, we complete the whole proof of Theorem \ref{Th1.5}.
\end{proof}

\begin{remark}
    For inequality \eqref{M4} in the case $n=kr$, we note that the equality condition is rather different from those in the case $n>kr$ and for the other inequalities, in that it admits one additional degree of freedom.
\end{remark}

Finally, we give the proof of Corollary \ref{Co1.5} based on the proof of Theorem \ref{Th1.5}.
  
\begin{proof} [Proof of Corollary \ref{Co1.5}]
    For any $x\in\mathcal{P}_r$, we need to prove that $\sigma_j(x)>0$ for any $1\leq j\leq k$. For any given $1\leq j\leq k$, we select any $(jr)$-tuple $I = \{i_1,...,i_{jr}\} \subset\{1,...,n\}$ with $i_1<\cdots<i_{jr}$. Let 
\begin{equation*}
    x^I := (x_1^I,...,x_{jr}^I)\in\mathbb{R}^{jr},
\end{equation*}
where $x_l^I = x_{i_l}$, $\forall 1\leq l\leq kr$, we have $x^I\in\mathcal{P}_r$, then analogous to the proof of \eqref{S}, we can show that $\sigma_j(x^I)>0$.
    
    Since $C_{n}^{jr}\cdot C_{jr}^j = C_{n}^j \cdot C_{n-j}^{jr-j}$, we can obtain 
    \begin{equation}
        \sum_{\substack{I\subset\{1,\ldots,n\}\\|I|=jr}} \sigma_j(x^I) = C_{n-j}^{jr-j} \sigma_j(x) > 0,
    \end{equation}
    then $\sigma_j(x)>0$ for any $1\leq j\leq k$. Thus, we have $x\in\Gamma_k$, as desired.
\end{proof}

\begin{remark} \label{remark 3.11}
    In fact, we can show that the inclusion $\mathcal{P}_r \subset \Gamma_k$ in Corollary \ref{Co1.5} is strict, namely $\mathcal{P}_r \subsetneq \Gamma_k$. To see this, we take 
    $$
    x^{\ast} = \big( -\frac{1}{k+\epsilon},\underbrace{0,...,0}_{n-k-1},\underbrace{1,...,1}_{k} \big),
    $$
    where $\epsilon>0$. Clearly, we have $x^\ast\notin\mathcal{P}_r$, since $n-k-1\geq kr-k-1\geq r-1$. Next, we show that $x^\ast\in\Gamma_k$. For any $1\leq j\leq k$, we have
    \begin{equation}
       \begin{aligned}
             \sigma_j(x^\ast) &= C_k^j - \frac{1}{k+\epsilon} C_k^{j-1} \\
             & = C_k^{j-1} \left( \frac{k-j+1}{j} - \frac{1}{k+\epsilon} \right) \\
             & \geq C_k^{j-1} \left( \frac{1}{j} - \frac{1}{k+\epsilon} \right) \\
             & > 0,
       \end{aligned}
    \end{equation}
    thus $x^\ast\in\Gamma_k$.
\end{remark}

Finally, as mentioned in the introduction, we give the following remark to show that not all Maclaurin means satisfy an inequality relation on $\overline{K^+}$ when $n\geq5$.

\begin{remark} \label{n geq 5} 
    For $n\geq5$, $k+l\leq n$ and $kl>n$, where $1<k,l<n-1$, we first show that the $k$-convex cone $\mathcal{P}_k$ and the G$\mathring{\text{a}}$rding's cone $\Gamma_l$ have no inclusion relation. To see this, we can take 
    $$
    x^\ast = (-(k-1)+\epsilon,\underbrace{1,...,1}_{n-1}),
    $$ 
    where $0<\epsilon<\frac{kl-n}{l}$. Clearly, we have $x^\ast\in\mathcal{P}_k$. Moreover, by direct calculation, we have
    \begin{equation*}
       \begin{aligned}
          \sigma_l(x^\ast) &= \left(-(k-1)+\epsilon\right) C_{n-1}^{l-1} + C_{n-1}^l \\
          &= C_{n-1}^{l-1} \left(-(k-1)+\epsilon+\frac{n-l}{l}\right) \\
          &= C_{n-1}^{l-1} \left(-\frac{kl-n}{l}+\epsilon\right) \\
          &< 0,   
       \end{aligned}
    \end{equation*}
    thus $x^\ast\notin\Gamma_l$. Moreover, similar to Remark \ref{remark 3.11}, we can take 
    $$
    x^{\ast\ast} = \big( -\frac{1}{l+\epsilon},\underbrace{0,...,0}_{n-l-1},\underbrace{1,...,1}_{l} \big),
    $$
    where $\epsilon>0$. Clearly, we have $x^{\ast\ast}\notin\mathcal{P}_{k}$, since $n-l-1\geq k-1$. Next, we show that $x^{\ast\ast}\in\Gamma_l$. For any $1\leq j\leq l$, we have
    \begin{equation}
       \begin{aligned}
             \sigma_j(x^{\ast\ast}) &= C_l^j - \frac{1}{l+\epsilon} C_l^{j-1} \\
             & = C_l^{j-1} \left( \frac{l-j+1}{j} - \frac{1}{l+\epsilon} \right) \\
             & \geq C_l^{j-1} \left( \frac{1}{j} - \frac{1}{l+\epsilon} \right) \\
             & > 0,
       \end{aligned}
    \end{equation}
    thus $x^{\ast\ast}\in\Gamma_l$.
    
    In this case, we can also show that the corresponding Maclaurin means $M_k$ and $E_{l}^\frac{1}{l}$ have no inequality relation on $\overline{K^+}$. Firstly, we take 
    $$
    x^\ast = (\epsilon,\underbrace{1,...,1}_{n-1}),
    $$
    where $\epsilon>0$. By direct calculation, we have
    \begin{equation*}
        M_k(x^\ast) = \left( \frac{\epsilon+k-1}{k} \right)^\frac{k}{n},
    \end{equation*} 
    and
    \begin{equation*}
        E_{l}^\frac{1}{l}(x^\ast) = \left( \frac{l\epsilon+n-l}{n} \right)^\frac{1}{l}.
    \end{equation*}
    Since $kl>n$, we have
    \begin{equation*}
     \frac{1}{l}-\frac{k}{n} = \frac{n-kl}{nl}<0,
    \end{equation*}
    then for sufficiently large $\epsilon > 0$, we have $M_k(x^\ast) > E_{l}^\frac{1}{l}(x^\ast)$. Moreover, if we take 
    $$
    x^{\ast\ast} = (\underbrace{0,...,0}_{k},\underbrace{1,...,1}_{n-k}),
    $$
    since $l\leq n-k$, by direct computation, we have
    \begin{equation*}
        M_k(x^{\ast\ast}) = 0 <   \left( \frac{C_{n-k}^l}{C_n^l} \right)^\frac{1}{l} = E_{l}^\frac{1}{l}(x^{\ast\ast}),
    \end{equation*}
    as desired.
\end{remark}

\vspace{3mm}

\section*{Acknowledgment}
This work is supported by the National Natural Science Foundation of China (No. 12271093), the Jiangsu Provincial Scientific Research Center of Applied Mathematics (No. BK20233002) and Shanghai Institute for Mathematics and Interdisciplinary Sciences (SIMIS) under grant number SIMIS-ID-2025-AD.


\bibliographystyle{plain}
\bibliography{ref}

\end{sloppypar}
\end{document}